\documentclass[letterpaper,11pt]{amsart}

\usepackage{amssymb,amsfonts}
\usepackage{url}
\usepackage[all,arc]{xy}
\usepackage{enumerate}
\usepackage{mathrsfs}
\usepackage{tikz-cd}
\usepackage{hyperref}

\newtheorem{thm}{Theorem}[section]  

\newtheorem{prop}[thm]{Proposition}

\theoremstyle{definition}
\newtheorem{defn}[thm]{Definition}

\theoremstyle{remark}

\makeatletter
\let\c@equation\c@thm
\makeatother
\numberwithin{equation}{section}

\newcommand{\R}{\mathbb{R}}
\newcommand{\Q}{\mathbb{Q}}

\newcommand{\Z}{\mathbb{Z}}
\newcommand{\C}{\mathbb{C}}

\newcommand{\g}{\mathfrak{g}}

\title{Scattering Theory on Higher $\Q$-rank Locally Symmetric Spaces}

\author{Punya Plaban Satpathy}

\date{}

\def\g{\mathfrak{g}}
\def\p{\mathfrak{p}}
\def\k{\mathfrak{k}}

\def\h{{\mathfrak h}}

\def\a{\mathfrak{a}}

\def\R{\mathbb{R}}
\def\C{\mathbb{C}}
\def\Z{\mathbb{Z}}
\def\Q{\mathbb{Q}}

\def\H{\mathbb{H}}

\def\M{\mathcal{M}}

\def\Aat{\mathcal{A}}
\def\T{\mathcal{T}}

\def\1{1\!\!1}

\def\gg{\mathfrak{g }}

\def\aa{\mathfrak{a }}

\def\1{1\!\!1}

\begin{document}

\begin{abstract}
In 1977, Victor Guillemin published a paper discussing geometric scattering theory, in which he related the Lax-Phillips Scattering matrices (associated to a noncompact hyperbolic surface with cusps) and the sojourn times associated to a set of geodesics which run to infinity in either direction.
This work was later extended to $\Q$-rank one Locally symmetric spaces coming from Semisimple Lie groups by Lizhen Ji and Maciej Zworski.
Here, we will extend some of the above mentioned results to higher rank locally symmetric spaces, in particular we will introduce higher dimensional analogues of scattering geodesics called \textbf{Scattering Flat} and study these flats in the case of the locally symmetric space given by the quotient
$SL(3,\Z) \backslash SL(3,\R) / SO(3)$. A parametrization space is discussed for such scattering flats as well as an associated vector valued parameter (bearing similarities to sojourn times) called \textbf{sojourn vector} and these are related to the frequency of oscillations of the associated scattering matrices coming from the minimal parabolic subgroups of $\text{SL}(3,R)$. The key technique is the factorization of higher rank scattering matrices.
\end{abstract}

\maketitle

\tableofcontents

\section{Introduction}
\label{intro}

The whole area of spectral analysis of differential operators on Manifold may be said to have begun, when Marc Kac announced his famous problem on \textbf{Hearing the shape of a Drum}, which in mathematically rigorous terms can be stated as follows. (Although Hermann Weyl was the first one to study the problem dealing with the spectral theory of bounded domains in $\mathbb{R}^d$ for $d = 2,3$).
  
Let $\Omega \subset \mathbb{R}^n$ be a bounded domain, then we are interested in solving the eigenvalue problem $\Delta \phi = \mu \phi $ on $\Omega$ (with either Dirichlet or Neumann boundary conditions). The spectral theorem asserts that there exists an orthonormal basis $f_n \in L^2(\Omega)$ consisting of (Dirichlet or Neumann) eigenfunctions, where the corresponding eigenvalues $\mu_1 \leq \mu_2 \leq \ldots \leq \mu_n \leq \ldots \longrightarrow \infty $ which accumulate only at $\infty $ and each have finite multiplicity.
  
\begin{prop}\label{prop:1_1}
Let $N(\lambda)$ be the Dirichlet eigenvalue counting function on a bounded domain $\Omega$. Then,
\begin{equation}
N(\lambda) \sim C_n [vol(\Omega)]\lambda^{n/2}
\end{equation}
where $C_n$ is a constant depending only on the dimension $n$ and on $\omega_n$, the volume of the unit ball in $\mathbb{R}^n$.
\end{prop}
  
Continuing along the above lines, an analogous result for a compact connected oriented manifold was obtained which we discuss now. $(M,g)$ be an n-dimensional compact connected oriented Riemannian manifold without boundary and let $\Delta$ denote the unique Friedrichs extension of the associated Laplacian on $M$ with respect to the metric $g$. Then it is a standard fact that, on the Sobolev space $H^2(M)$, the extended Laplacian $\Delta$ has a positive discrete spectrum with eigenvalues $\lambda_n \longrightarrow \infty$ as $n \longrightarrow \infty$.
 
\begin{prop}\label{prop:1_2}
Let $(M,g)$ be as above, and define the set
\begin{equation*}
N(\lambda) = \{\lambda_i \leq \lambda \in \R \hspace{0.1cm}|\Delta \phi_i  = \lambda_i \phi_i \hspace{0.1cm} for \hspace{0.1cm} some \hspace{0.1cm}  \phi_i \neq 0 \in L^2(M) \}
\end{equation*}
 
Then we have, $N(\lambda) \sim C_n [vol(M)]\lambda^{n/2}$ where $C_n$ only depends on the dimension of the manifold $M$
and on $\omega_n$, is the volume of the unit ball in $\mathbb{R}^n$.
\end{prop}

For a proof, refer to Theorem 3.27 in \cite{Rosenberg1997}.

Note that this result about the asymptotics of the counting function of eigenvalues is, in fact, a result of the trace formula associated with the heat kernel of the Laplacian. The motivation for such trace formulas of course came from the classical Poisson summation formula which states that for $f \in C_0^{\infty}(\R^n)$ with the associated Fourier transform $\hat{f}$, we have

\begin{equation*}
\sum_{k=-\infty}^{\infty}f(k) = \sum_{k=-\infty}^{\infty}\hat{f}(k)
\end{equation*}
Refer to Theorem 0.1.16 in \cite{Sogge1993} for the proof and discussion of its connection to the spectral theory of the $\text{n}$ dimensional torus.

The first significant breakthrough in producing a Poisson type formula was achieved by Atle Selberg when he announced his famous Trace formula, which encodes the spectral data as well as the global geometry of a compact hyperbolic surface.

\begin{prop}
Let $X$ be a compact hyperbolic surface with the Laplacian $\Delta$ and the associated discrete spectrum given by the set
$\{\mu_i\}_{i=1}^{\infty}$. Denote by $\mathcal{T}$ the set of lengths of primitive closed geodesics in $X$. Then for $t>0$,

\begin{equation*}
    \sum_{i=1}^\infty e^{-t\mu_i} = \text{Area}(X)\frac{e^{-t}}{(4 \pi t)^{3/2}}\int_{0}^\infty \frac{re^{-r^2/4t}}{\text{Sinh}(r/2)}dr + \frac{e^{-t}}{(4 \pi t)^{1/2}}\sum_{\tau \in \T} \sum_{k=1}^{\infty} \frac{\tau}{2\text{Sinh}(k\tau/2)}e^{-k^2\tau^2 / t^2}
\end{equation*}
\end{prop}

For a detailed discussion refer to \cite{Borthwick2016}.

Since then, several attempts have been made to generalize Selberg's formula to more general Riemannian manifolds. The most notable generalization is the Duistermaat-Guillemin distributional trace formula\cite{Duistermaat1975}, which 
relates the spectral geometry(in terms of $L^2$ eigenvalues of the Laplacian) on a compact manifold to the geodesic flow dynamics on the tangent bundle in terms of the length spectrum of closed geodesics.

The scenario is rather complicated if the underlying manifold $M$ is noncompact; this would create the possibility for the existence of both a discrete as well as a continuous spectrum for the Laplacian. One can show that the associated continuous spectrum will not change under any compact perturbations of $M$ and so as a set it depends only on the geometry at infinity. For a discussion on the decomposition principle refer to Proposition 2.1 in \cite{Donnelly1979}.

     Consider a hyperbolic surface of fixed genus having a finite number of inequivalent cusps along with a finite area. Then the spectrum of the Laplacian on such a surface has both discrete and as well as a continuous spectrum, which will be discussed extensively in chapter 3.
    
    For such a hyperbolic surface, it is possible to write down an analog of the Selberg trace formula. There are several ways to achieve this; one such method is Geometric scattering theory. Peter Lax and Ralph Phillips used this framework to study the automorphic wave equation on such a noncompact hyperbolic surface and constructed Eisenstein series along with an associated set of scattering matrices as part of their proof of the heat trace formula.
    
    For details, refer to \cite{Lax1977}.
    
   In 1976 Victor Guillemin wrote a paper in which he investigated geometric scattering theory for several scenarios, one of which was the case of a noncompact hyperbolic surface as above.
   
   In order to give a geometric context to scattering matrices he defined a certain class of scattering geodesics and showed a correspondence that was very similar to a Poisson type formula. The key idea was to use the result of Lax and Phillips which expressed the scattering matrix in terms of certain horocyclic integrals of Eisenstein Series.

     An important class of noncompact Riemannian manifolds with a non-empty continuous spectrum is given by finite volume noncompact locally symmetric spaces of the form $M = \Gamma \backslash G / K$, where $G$ is a semi simple lie group, K one of its maximal compact subgroups and $\Gamma$ being a neat finite co-volume arithmetic subgroup of $G$. The geometry at infinity of such a manifold $M$ is completely described by the reduction theory of $G$ with respect to the $\Gamma$ conjugacy classes of rational parabolic subgroups.

    The reduction theory tells us that $M$ then can be decomposed as a union of a compact core $M_c$ and a certain number of noncompact Siegel sets each associated rational parabolic subgroups of $G$, for details refer to the precise reduction theory in (add reference).

     For such a locally symmetric space $M$ one can show the existence of scattering geodesics (look for the details in section 6.2),  which run from one end to another and are eventually distance minimizing in either direction.
    These scattering geodesics spend a finite amount of time in the compact core $M'$ and this is the \textbf{sojourn Time} associated to the scattering geodesics.

 One of the earliest results about the relations between these sojourn times and the associated scattering matrices related to the locally symmetric space $M$ was done by Lizhen Ji and Maciej Zworski in \cite{Ji2001}. This is where the spectral decomposition of $M$ comes into the picture, it was completely worked out by Robert Langlands, essentially the discrete spectrum is characterized by a set of $L^2$ eigenfunctions, and the continuous spectrum is characterized by the \textbf{Eisenstein Series} associated to rational parabolic subgroups of $G$ as well as associated cusp forms on the boundary locally symmetric spaces. The main idea of proof in \cite{Ji1999} is  very similar to that of Guillemin, which involves expressing the scattering matrix as a horocyclic integral of an Eisenstein series from the work of  Langlands and Harishchandra .
    
    This is the approach we will use to explore certain aspects of geometric scattering theory on higher rank locally symmetric spaces.
    
 Here is the outline of for the rest of the paper, section two two gives a quick review  Guillemin's\cite{Guillemin1976} work on scattering geodesics in non-compact hyperbolic surfaces.

  In section three and four we review the geometry at Infinity of a general locally symmetric space, discuss the associated spectral resolution of the Laplacian along with the scattering matrices. Then in section five, we review the results from \cite{Ji1999} about scattering geodesics and the latter sections introduce the notion of scattering flats, certain higher dimensional analogue of scattering geodesics. These scattering flats are shown to have certain associated sojourn vectors and the result involves studying the singular support of higher rank scattering matrices and relating that to the sojourn vectors. A key tool is the result of Harishchandra \cite[Chapter 2, Section 5]{HarishChandra1968} involving factorization of higher rank scattering matrices .

We now state the main results of this paper. 

\begin{thm}
    Let $\Gamma \backslash X$ be a locally symmetric space of rational rank equal to $\text{q}$, where $X = G/K$, with $G$ being the real locus semi-simple complex linear algebraic group $\boldsymbol{G}$, along with a maximal compact subgroup $K \subset G$. We further assume that $Rank_{\Q}(\boldsymbol{G}) = Rank_{\R}(\boldsymbol{G}) $.
    
    I)For any two minimal associate distinct rational parabolic subgroups of $\boldsymbol{G}$ and for any $\gamma \in \Gamma$, there is a family of scattering flats in $\Gamma \backslash X$ with a common sojourn vector which only depends on $\gamma$.
    Furthermore, this family of scattering flats is parametrized by a common finite cover of the boundary locally symmetric spaces associated to the two rational parabolic subgroups(Theorem \autoref{theorem:6_4}).
    
    II)Any such family of scattering flats project onto a family of scattering geodesics in a certain associated boundary locally symmetric space $S_P$, corresponding to a rational parabolic subgroup $P$ such that the sojourn time associated to this family of scattering geodesics is precisely the norm of the sojourn vector(Theorem \autoref{theorem:8_5}) .
\end{thm}

\begin{thm}
    Let $\text{X} = SL(3,\Z) \backslash SL(3,\R) /SO(3)$ be the rational rank two locally symmetric space. Let $P_0$ be the minimal parabolic subgroup of $SL(3,\R)$ consisting of upper triangular matrices. For a chosen $Id \neq w \in S_3$, let $C(w,\lambda)$ denote the rank two scattering matrix corresponding to $P_0 $ with $\lambda =(\lambda_1,\lambda_2,\lambda_3) \in \C^3$ and $Re(\lambda_i ) >> 0$ then for  $\eta =(\eta_1,\eta_2,\eta_3) \in R^3$ and $w = (12) \in S_3$, the Singular support of the generalized Fourier transform of $C(w,i\eta)$ is precisely the set $\{(T,T,0) \in \R^3\}$, where $T \in \T$ with $\T$ being the set of sojourn times associated to scattering geodesics on the hyperbolic surface $SL(2,\Z) \backslash \H$ running between the unique cusp end at $\infty$ to itself(Theorem \autoref{theorem:9_21}).
\end{thm}
This result does extend to other $w \neq Id \in S_3$.

\begin{thm}
    Let $\text{X} = SL(3,\Z) \backslash SL(3,\R) /SO(3)$ be the rational rank two locally symmetric space. Let $P_0$ be the minimal parabolic subgroup of $SL(3,\R)$ consisting of upper triangular matrices. 
  
    I) For any $\gamma \in SL(3,\Z)$ such that $\gamma \notin P_0$, there is a continuous family of scattering flats in $X$ parametrized by the space of upper triangular unipotent matrices $N$ in $SL(3,\R)$ and all of these scattering flat have the same \textbf{sojourn vector} which only depends on $\gamma$.
    We further have that any such given family of scattering flats projects onto a family of scattering geodesics into an associated Locally symmetric space which can be naturally identified with $SL(2,\Z) \backslash \H$ with a common sojourn time given by the norm of the sojourn vector.

    II) Denote by the set of such sojourn times of scattering geodesics in $SL(2,\Z) \backslash \H$ as $\T$ and 
    $C(s)$ the unique scattering matrix for $SL(2,\Z) \backslash \H$ associated to the cusp at $\infty$. If $C(w,\lambda)$ denotes one of the rank two scattering matrices associated to $\Gamma \backslash X$ as in Chapter 6 with $\lambda = (\lambda_1,\lambda_2,\lambda_3) \in \C^3$ with $Re(\lambda) >> 0$. Then $C(w,\lambda)  = C(\tilde{s})$, with $\tilde{s}$ depending linearly on
    $\lambda$ and furthermore $C(w,\lambda)$ admits an asymptotic expansion where the terms only depend on the sojourn times $T \in \T$(Theorem \autoref{theorem:9_22}).
\end{thm}

\pagebreak

\section{Scattering on Finite Area Hyperbolic Surfaces}

Victor Guillemin was the first one who realized that just as in the case of a compact manifold, there
is an analogous Poisson relation between sojourn times of scattering geodesics and singularities of
the scattering matrices for noncompact hyperbolic surfaces. Here we will review the work of
Guillemin's paper (see \cite{Guillemin1976} for details)

Let $\H = \{z = x + iy | x, y \in R, y > 0\}$ be the upper half plane with the assigned hyperbolic metric
$ds^2 = \frac{dx^2 + dy^2}{y^2}$. The associated Laplacian is given by 
$\Delta = -y^2(\frac{\partial^2}{\partial x^2 } + \frac{\partial^2}{\partial y^2 }) $, since $\H$ is a complete
Riemannian manifold,  the Laplacian has a unique self-adjoint extension which is also be denoted by $\Delta$.
Now consider a cofinite discrete torsion free subgroup $\Gamma$ of PSL(2,$\R$) and let $X = \Gamma \backslash \H$ be the
associated finite area noncompact hyperbolic surface with $k_1, ...,k_n$ inequivalent cusps. One then
knows that, for any sufficiently large ''$a$'' , X is a disjoint union of compact subset $X_a$ and a finite
number of open sets $X_i$ , $i = 1, 2 \dotsc n$ , where $X_i$ is the cusp neighborhood for the corresponding
cusp $k_i$ and so that each $X_i$ is isometric to the set  $\{ -1/2 \leq Re(z) \leq 1/2 | |Im(z) | \geq a \}$ in the upper
half plane. We will fix such an $a$ for the rest of this chapter.

It is a well-known fact that any geodesic in $\H$ is either a half line of the form $Re(z) = k$, with $Im(z) > 0$,
 or a half circle of the form $|z-p | = q$, with $p \in \R$, $q \in \R^{+}$ and $Im(z) > 0$. Let $\pi : \H \longrightarrow X$ be the
canonical projection, then the hyperbolic metric on $\H$ introduces a natural hyperbolic metric on
$X$ with respect to which $\pi$ becomes a local isometry, further we have that geodesics in $\H$ project
to geodesics in $X$ under $\pi$. We are now all set to describe a certain class of geodesics in $X$, that run to infinity in either direction, in the sense that they are not contained in any compact subset of $X$.

\begin{defn}\label{def:3_1}
A geodesic $\gamma(t)$ in $X$ is called a scattering geodesic if it is contained in $X \backslash X_a$ for large positive as
well as negative times $t$. A scattering geodesics that is contained in $X_i $ for $t \ll t_0$ and in $X_j$ for $ t_1 \ll t $ is called a geodesic scattered between cusp ends $X_i$ to $X_j$. The associated \textbf{sojourn time} $T_{\gamma}$
is the total amount of time the geodesic spends in the compact core $X_0$, starting from the first time
it entered $X_0$ until the time when it exits.
\end{defn}

Fixing  an ''$a$'' as before, Guillemin then proved that there are countable number of non-trivial scattering
geodesics running between cusp ends $k_i$ and $k_j$ and computed the sojourn time for such geodesics, we will give a sketch of the proof and the details can be found at \cite{Guillemin1976}.

Start with a fundamental domain of $\Gamma$ in $\H$, such that the cusp $k_i$ is at $\infty$ and the cusp $k_j$ will then be a vertex of the fundamental domain lying on the real axis given by the point $(x_j,0)$ with the cusp neighborhood two geodesics $\sigma_1$ and $\sigma_2$ which are perpendicular to the real axis at $(x_j,0)$.

Now choose a geodesic $\sigma$ in $\H$ joining the cusp at $k_j$ and a point $q$ on the real axis, such that $\sigma$ lies between $\sigma_1$ and $\sigma_2$. (Look at the figures below)	Then for any element $B \in \Gamma$, $B\sigma$ is either a half circle with center on the real axis or a \textbf{half line perpendicular to the X-axis}. In the latter case, the projection of $\sigma $ onto $X$ is going to be a scattering geodesic running between cusp ends from $k_i$ to $k_j$. Note that not every
$B \in \Gamma$ gives rise to such a scattered ray, and for any B there is at most one scattering geodesic,
since such a choice of B forces $q$ to be $B^{-1}(\infty)$. Finally, note that since $\Gamma$ is a discrete subgroup of
$PSL(2,\R)$, it is countable. Hence, the set of $B \in \Gamma$ which give rise to scattering geodesics is also
countable.

\begin{figure}
    \centering
    \includegraphics[width = 1.0\textwidth]{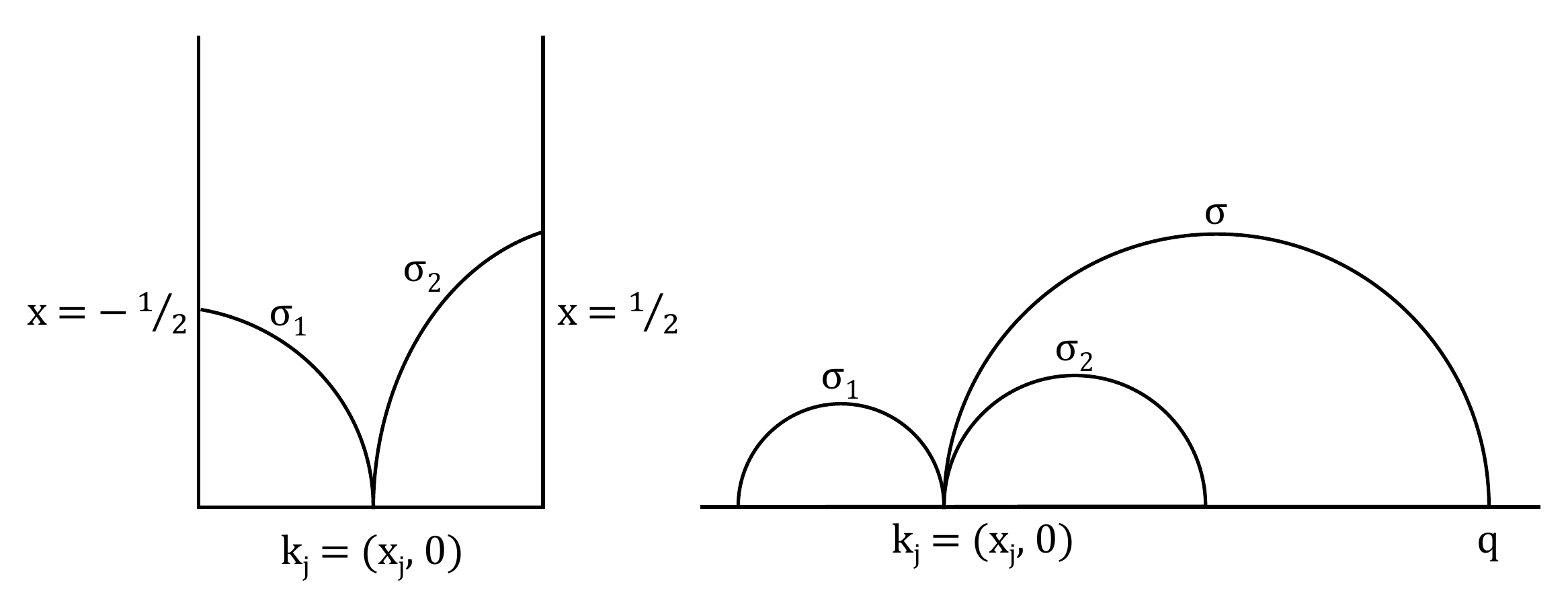}
    \caption{Construction of scattering geodesics.}
\end{figure}

We are now ready to state and prove the main result of Guillemin. As stated previously, the cusp $k_i$  is at $\infty$  and the cusp $k_j$ will then be a vertex of the fundamental domain lying on the real axis given by the point $(x_j,0)$ with the cusp neighborhood bounded by two geodesics $\sigma_1$ and $\sigma_2$ which are perpendicular to the real axis at $(x_j,0)$.  Also, note that as stated earlier if $B \in \Gamma$ gives rise to a scattering geodesic, then the transformations $\gamma B$ also generate the same scattering geodesic in $X$ (according to the procedure given in the previous section) for every $\gamma \in \Gamma _{\infty}$, where $\Gamma _{\infty} \subset \Gamma $  is the subgroup generated by the transformations $z \mapsto z+1$. As before, we choose an isometry $\Psi$ mapping the vertical strip $\{-1/2 \leq Re(z) \leq 1/2 \}$ onto the $j$-th cusp neighborhood such that $\Psi(\infty) = k_j$.

Associated to the $i$-th cusp, we have the \textbf{Eisenstein Series} $E_{\infty}(z,s)$ given by,

\begin{equation}
E_{\infty}(z,s) = \sum_{B \in \Gamma_{\infty} \backslash \Gamma} (Im(Bz))^s
\end{equation}

The Eisenstein Series satisfies the following properties,

\begin{enumerate}
    \item $E_{\infty}(z,s)$ converges uniformly and absolutely on compact subsets of the half plane $Re(s) > 1$ and defines a holomorphic function, as well with a meromorphic continuation to all $s \in \C$ and is regular on $Re(s) = 1/2$.

    \item $E_{\infty}(\gamma z,s)  =E_{\infty}(z,s) $ for all $\gamma \in \Gamma$ .

    \item As a function of $z$, $E_{\infty}(z,s) $ is smooth and denoting by $\Delta_z$ the hyperbolic Laplacian in the $z$-variable, $E_{\infty}(z,s) $satisfies the equation $(\Delta_z -s(1-s))E_{\infty}(z,s)  = 0$.
\end{enumerate}

Now we set $s = 1/2 +i \tau$, and let $E(z, \tau) = E_{\infty}(z,1/2+i \tau) $. Then observe that the zero-th Fourier coefficient in the expansion of $E(z,\tau)$ in the $j$-th cusp neighborhood is given by the integral,

\begin{equation*}
    \int_{-1/2}^{1/2} E(\Psi z,\tau) dx 
\end{equation*}

In their work, Lax and Phillips \cite[Chapter 8]{Lax1977} showed that:

\begin{equation}
    e^{-2i \tau ln(a)}\int_{-1/2}^{1/2} E(\Psi z,\tau) dx  = C_{ij}(\tau) y^{1/2 -i \tau}
\end{equation}

Where, $C_{ij}(\tau) $ is the $ij$-th entry of the scattering matrix,this scattering matrix contains important
information about the geometry of the surface X, in particular the determinant of scattering matrix
shows up in the Selberg trace formula for finite area noncompact hyperbolic surfaces with cusps.

We are now ready to state the main result of Guillemin.

\begin{thm}\label{theorem:2_4}\cite[Theorem 3]{Guillemin1976}
Let $\T_{ij}$ be the set of sojourn times for geodesics in $\Gamma \backslash \H$ that are scattered from the i-th cusp neighborhood to the j-th cusp neighborhood.  Define the following integral,

\begin{equation}
F(\tau) = \int_{-\infty}^{\infty} (1+w^2)^{-(1/2+i \tau) } dw
\end{equation}
Then for $Im(\tau) \leq -3/2$, one has

\begin{equation}
C_{ij}(\tau)  = aF(\tau) \sum_{T_{\sigma} \in \T_{ij}}e^{-T_{\sigma}(1/2+i \tau)}
\end{equation}
For a general $\tau$, the right-hand side is supposed to be the meromorphic continuation of this series
\end{thm}

\pagebreak
\section{Locally Symmetric Spaces and their Geometry at Infinity}
\label{symspacegeometry}

Let $\mathbf{G}$ be a complex semisimple linear algebraic group. Denote by $G$ the real locus $\mathbf{G}(\R)$ of $\mathbf{G}$, then $G$ is a real semisimple lie group with finitely many connected components. Choose a maximal compact subgroup $K$ of $G$. (Note that all such maximal compact subgroups are conjugate to each other). Then the associated symmetric space $X = G \backslash K$
is a negatively curved Riemannian manifold which admits a $G$-action as well as a  $G$-invariant Riemannian metric, which we now define.

Let $\g$ be the Lie algebra of $G$, then $\g$ admits a Cartan decomposition $\g = \k \oplus \p$, where $\k$ is the Lie algebra of $K$. If $B(\bullet, \bullet)$ denotes the associated Killing form of $\g$, then it is known that the restriction of $B$ to $\k$ is negative definite and its restriction to $\p$ is positive definite.

Now let's fix the base point $x_0 \in X$ given by the coset $eK$, where $e \in G$ is the identity element. Then their is a natural identification of $T_{x_0}X $ with $\p$ as vector spaces and the restriction of the Killing form $B(\bullet, \bullet)$ to $\p$ defines a positive definite inner product on $T_{x_0}X$. for any other point $x \in X$, one can define a $G$ -invariant inner product with varies smoothly, and is essentially obtained by recognizing that tangent spaces at different points of $X$ are isomorphic with the isomorphism given by a group action of $G$.

Now let $\Gamma$ be  an arithmetic subgroup of $\mathbf{G}$ with finite covolume, for technical reasons we will assume that $\Gamma$ is neat, i.e. for any embedding of $\Gamma$ into $SL(n,\Z) $, an  arbitrary element $\gamma \in \Gamma$  has no roots of unity as an eigenvalue except 1. Such a neat subgroup $\Gamma$ always admits a torsion free subgroup, so without loss of generality we can assume that 
$\Gamma$ is torsion free. 

Define the Riemannian manifold $S = \Gamma \backslash X = \Gamma \backslash \ G / K$.  Then, $S$ is a locally symmetric space of finite volume. When $S$ is noncompact(which is the case we will focus on), it has both discrete and continuous spectrum, the latter characterized by Eisenstein series and a certain set of intertwining operators called the \textbf{scattering matrices}.

  \subsection {Rational Langlands Decomposition}
Let $\mathbf{P}$ be a rational parabolic subgroup of $\mathbf{G}$, and let $P$ be the real  locus $\mathbf{P}(\R)$.  Then $X$ has a natural horospherical decomposition associated to $P$, which we describe now.

Let  $\mathbf{N}_P$ the unipotent radical of $\mathbf{P}$, define  
 $N_P = \mathbf{N}_P(\R)$, further note that $\mathbf{H}_P$ (called the Levi quotient)  given by $\mathbf{P} / \mathbf{N}_P$ is an algebraic group defined over $\Q$, let $\mathbf{C}_P$  be the split center of $\mathbf{H}_P$ defined over $\Q$ and further define $A_P$ to be the connected component of identity in the real locus of $\mathbf{C}_P$. Finally, let $\mathbf{H}_P^* $ denote the set of rational characters of $\mathbf{H}_P$, and define the reductive group $\mathbf{M}_P = \cap _{\eta \in \mathbf{H}_P^*}$ Ker $\alpha ^2$ and define $M_P = \mathbf{M}_P(\R)$. The subgroup $A_P$ will be called a rational split component of the parabolic subgroup $\mathbf{P}$ and  denote by $\Sigma^+ (P,A_P) $ a set of positive roots  (with respect to some ordering) corresponding to the adjoint action of $A_P$ of $N_P$.

\begin{prop}\label{prop:3_1}
Let $P$ be the real locus of a rational parabolic subgroup $\mathbf{P}$ of $\mathbf{G}$, then $P$ admits a rational Langlands decomposition given by the diffeomorphism,

\begin{equation}
M_P \times A_P \times N_p  \simeq P \hspace{0.1cm}, \hspace{0.2cm} (m,a,n) \mapsto man 
\end{equation}

\end{prop}
 
 Note that, the dimension of $A_P$ is called the $\Q$-rank of $\mathbf{P}$, the {$\Q$-rank of $\Gamma \backslash X$} is the maximum possible $\Q$-rank of a rational parabolic subgroup $\mathbf{P}$ of $\mathbf{G}$.

 We now discuss the associated Horospherical decomposition  for $S$. Note that, since $P$ acts transitively on $X = G /K$, the Langlands decomposition of $P$ described in the above theorem gives rise to a splitting of the symmetric space $X \simeq M_P /(K  \cap P) \times A_P \times N_P$, this is proved using the fact that $A_P$ and $M_P$  commute along with the fact that $K \cap P$ is a maximal compact subgroup of $M_P$. We denote by $X^P$ the manifold $M_P  /(K  \cap P) $, which is   the boundary symmetric space of $X$ associated to the rational parabolic subgroup P.
 
 Suppose $\a_P$ is the Lie algebra of $A_P$, along with the diffeomorphic  exponential mapping $exp_P: \a_P \rightarrow A_P$. Then any $x \in X$ can be represented as $x = (z,exp_P(H),w) \in X^P \times A_P \times N_P$ and $H \in \a_P$.

 Under the natural projection map $\mathbf{P} \rightarrow \mathbf{H}_P$, the subgroup $\Gamma \cap P$ is mapped to an arithmetic subgroup of $\mathbf{H}_P$ which is in fact contained in $\mathbf{M}_P$, denote this image by $\Gamma^P$. Then, $\Gamma^P$ acts properly discontinuously on $X^P$ and the quotient has finite volume and is a manifold due to the fact that $\Gamma$ is assumed to be neat which results in $\Gamma^P$ being torsion free.  The resulting locally symmetric space $S_P = \Gamma^P \backslash X^P$ is the associated boundary component  of $S$ corresponding to the rational parabolic subgroup $\mathbf{P}$.

 \subsection{Reduction theory of Locally Symmetric Spaces}
 
 Since $S$ is assumed to be noncompact it has a natural decomposition into a compact core and a finite number of noncompact Siegel sets which are build out of and are in one-one correspondence with a set of representatives of $\Gamma$-conjugacy classes of rational parabolic subgroups of $G$.
 
 Let $Q_1,Q_2,...,Q_m$ be the set of representatives of $\Gamma$-conjugacy classes of $\Q$-rank one rational parabolic subgroups of $\mathbf{G}$. Denote by $\a_i$ the Lie algebra of the rational split component $A_{Q_i}$ of $Q_i$ and introduce the direct sum $\mathcal{A} = \displaystyle \oplus_{i=1}^m \a_i$. Then for any rational parabolic subgroup $Q$ of $\mathbf{G}$,  there is a well-defined map
 $\Psi_Q: \Aat \longrightarrow A_Q$, such that for $Q = Q_i$, $\Psi_{Q_i}$ is the projection map onto the i-th component $\Aat$ which is $\a_i$.

 Let $\tau_i$ be half the sum of the roots in $\Sigma(Q_i,A_{Q_i})$ with multiplicity. There is a unique element $T_i \in \a_i$ corresponding to $\tau_i$ obtained by using the Killing form. Using the map above,  one can show the existence of a unique $T \in \Aat$ such that $\Psi_{Q_i}(T) = T_i$.
 
 For $r >> 0$ and $T_r = rT$, one then defines a shift of the positive Weyl chamber $\a_Q^{+}(T_r)$ (defined with respect to $\Sigma^+(Q,A_Q)$)  given by,
 
\begin{equation}
\a_Q^+(T_r) = \{h \in \a_Q \hspace{0.1cm} |  \hspace{0.1cm}  \beta(h) > \beta (\Psi_Q(T))  \hspace{0.1cm}  \forall \beta \in \Sigma^+(Q,A_Q)   \hspace{0.1cm} \}
\end{equation}
 also define ,
 
 \begin{equation}
A_{Q,+}(T_r) = \{exp_Q(h) \in A_Q \hspace{0.1cm} |  \hspace{0.1cm}  \langle \Psi_Q(T) -h , w \rangle \geq 0  ,   \hspace{0.1cm} \forall w \in \a_Q^{+}, h \in \a_Q \}
\end{equation}
Where, $\a_Q^{+}$ is the positive Weyl chamber in $\a_Q$ with respect to the set of positive roots $\Sigma(Q,A_Q)$.

 Denote by $X(T_r)$ the intersection $\displaystyle \cap_{Q} X_Q \times A_{Q,+}(T_r) \times N_Q$ over all proper rational parabolic subgroups of $\mathbf{G}$. Then $X(T_r)$ is a $\Gamma$-invariant submanifold  $X$ with corners of same dimension that of $X$.  Denote by $X_r  = \Gamma \backslash X(T_r)$.

 We are now ready to state the main result of the reduction theory of $S$.
 \begin{prop}\cite[Proposition 3.2.2]{Ji2002}\label{prop:3_5}
 Let $\mathbf{R}_1,\mathbf{R}_2,...,\mathbf{R}_d$ be a set of representatives of $\Gamma$-conjugacy classes proper rational parabolic subgroups of $\mathbf{G}$. Then, for every $1 \leq j \leq d$ there exists left $N_{R_j}$-invariant compact submanifold
 $\Omega_j$ with corners  in $X_{R_j} \times \Gamma_{R_j} \backslash N_{R_j}$ such that for $T_r = rT$ with $r >> 0$  the subset $\Omega_j \times exp_{R_j}(\a^+_{R_j}(T_r) ) \subset (\Gamma \cap R_j) \backslash X$ maps injectively into $\Gamma \backslash X$, denoting the image by still $\Omega_j \times exp_{R_j}(\a^+_{R_j}(T_r) )$, we have the following disjoint decomposition.
 
 \begin{equation}
     \Gamma \backslash X  = X_r \cup \bigsqcup _{j=1}^d \Omega_j \times exp_{R_j}(\a^+_{R_j}(T_r) )
 \end{equation}
 \end{prop}
 
 For future reference, we will call $\Omega_j \times exp_{R_j}(\a^+_{R_j}(T_r) )$ the Siegel end of $\Gamma \backslash X$ associated to the parabolic subgroup $\mathbf{R}_j$(and all of its $\Gamma$ conjugates).

\section{Spectral theory of Locally symmetric spaces}
\label{specdecomp}

In this section we will review the decomposition of the space $L^2( \Gamma \backslash X)$  under the action of the Laplace operator $\Delta$. (There is a corresponding decomposition that can be shown  to be invariant under all $G$-invariant differential operators, but we will focus on the Laplacian $\Delta$).

Since the space $X$ is noncompact with finite volume, it has both discrete as well as continuous spectrum, the continuous spectrum is  parametrized by the so-called \textbf{Eisenstein Series} which we will describe in this section.

In general, depending on the rational rank of a rational parabolic subgroup for $G$, there are multiple such Eisenstein series associated to the locally symmetric space $X$ which we will now describe.

The general procedure for constructing Eisenstein is as follows, let $Q$ be a rational parabolic subgroup of 
$G$ with the associated Langlands decomposition $Q = M_Q \times A_Q \times N_Q$, where $A_Q$ is the associated split component of $Q$, with Lie algebra $\a_Q$ and the associated exponential map $exp_Q : \a_Q \longrightarrow A_Q$. Further, let $X^Q$ be the associated boundary  symmetric space given by

$M_Q / (K \cap M_Q)$  with $\pi_Q : X\longrightarrow X^Q$ be the projection map coming from the associated horospherical decomposition. Further, under the horospherical decompsition denote a point $x \in X$
as $(n_x,exp_Q(H_Q(x)),z_x)$ with $n_x \in N_Q$, $z_x \in X^Q$ and $H_Q(x) \in \a_Q$. Finally, let $\tau_Q$ be half the sum of positive roots corresponding to the adjoint action of $\a_Q$ on $n_Q = Lie(N_Q)$.  

Now choose an $L^2$-eigenfunction $\psi$ of the Laplacian on the associated boundary locally symmetric space $S_Q = \Gamma_Q  \backslash X^Q$ and choose $\lambda \in \a_Q \otimes_{\R} \C$ such that
$Re(\lambda) \in \tau_Q+ \{\mu \in \a_Q^* | \langle \mu , \beta \rangle > 0 \vspace{0.1cm} \forall \beta \in \Sigma^{++}(Q,A_Q)\}$. Then we define the Eisenstein Series, $E_Q(x,\lambda,\psi)$, a smooth $\Gamma$ invariant function on $X$, which will then descend to a smooth function on the space $\Gamma \backslash X$.

\begin{equation}
E_Q(x,\lambda,\psi) = \sum_{\gamma \in (\Gamma \cap P) \backslash \Gamma} e^{(\tau_Q+\lambda)(H_Q(\gamma x))}\psi(\pi_Q(\gamma x)) 
\end{equation}

 If we have another parabolic subgroup $Q'$(not necessarily different from $Q$) with the associated Langlands decomposition 
 $Q' = M_{Q' } \times A_{Q'} \times N_{Q'}$, then the restriction of $E_Q(x,\lambda,\psi)$ along the parabolic subgroups $Q'$ is given by the integral.

 \begin{equation}
 E_{Q|Q'} (x,\lambda , \psi) = \int_{(\Gamma \cap N_{Q'}) \backslash N_{Q'}} E_{Q}(nx,\lambda,\psi) dn
 \end{equation}
 Where, $dn$ is a normalized Haar measure on $N_{Q'}$ so that the total volume of $N_{Q'}$ is equal to one.
 One of the main results of Langlands~\cite{Langlands1976} is the study of the restriction of Eisenstein Series along various parabolic subgroups of $G$. There are essentially two separate cases to be considered.

\begin{itemize}
 \item[$\bullet$] If $rank(Q) \geq rank(Q')$, and the two parabolic subgroups are not associate, then a result of Langlands says that $E_{Q|Q'} (x,\lambda , \psi)  = 0$.
 
 \item[$\bullet$] If the two parabolic subgroups are associate, then let  $S(Q,Q')$ denote the maps of the form $Ad(y)$,$y \in \mathbf{G}_{\Q}$  such that $Ad(y)A_Q = A_{Q'}$. Then the restriction of $Ad(y)$ to $\aa_Q$ will still be denoted by $Ad(y)$. By a theorem of Langlands \cite[Theorem 5, Chapter 2]{HarishChandra1968} , 
 $E_{Q|Q'} (x,\lambda , \psi) $ is given by the following sum.
  \end{itemize}
 \begin{equation}
 E_{Q|Q'} (x,\lambda , \psi)   = \displaystyle \sum_{s \in S(Q,Q')} e^{(\tau_{Q'}+s\lambda)H}(C_{Q|Q'}^{\mu}(s,\lambda)\psi)(\pi_{Q'}(x))
 \end{equation}
 
 Where, $x \in X$ is be represented as $x = (z,exp_{Q'}(H),w) \in M_{Q'} \times A_{Q'} \times N_{Q'}$ and $H \in \a_{Q'}$. The maps $C_{Q|Q'}^{\mu}(s,\lambda)$ are a set of meromorphic maps depending on the parameter $\lambda$, and mapping $\psi \in L^2(S_Q)$ (with eigenvalue $\mu$) to 
 $C_{Q|Q'}^{\mu}(s,\lambda)(\psi)  = \psi_1 \in L^2(S_{Q'})$. Note that, $\psi_1$ also ends up being an eigenfunction for the Laplacian, however this is a non-trivial fact and requires the knowledge that $\psi$ is actually a joint eigenfunction of all the $G$-invariant differential operators acting on $L^2 (\Gamma \backslash X)$.

Note that the scattering matrices $C_{Q|Q'}^{\mu}(s,\lambda)$
 clearly depend on the choices for the split component of $Q$ and $Q'$ respectively. So, anytime we talk about such scattering matrices they would be with respect to preassigned split components. The first thing we want to understand is how these matrices change if we choose a different pair of split components. We proceed as follows,

 We will be working with associate pairs of rational parabolic subgroups $\boldsymbol{Q},\boldsymbol{Q}'$ of $\boldsymbol{G}$ with the real locus being $Q,Q'$ and preassigned split components $A_Q,A_{Q'}$ respectively, choose $\gamma_1, \gamma_2  \in 
 \Gamma$ along with $p \in Q$ and $p' \in Q'$. Define parabolic subgroups $Q_1  = \gamma_1 Q \gamma_{1}^{-1}$ and $Q_2  = \gamma_{2} Q' \gamma_{2}^{-1}$, with the associates split component $A_{Q_1} = \gamma_1 A_Q \gamma_1^{-1}$ and 
$A_{Q_2} = \gamma_2 A_{Q'} \gamma_2^{-1}$ respectively. The following propositions summarize the transfomation of scattering matrices 

\begin{prop}\label{prop:4_4}
For $s \in S(Q,Q')$ and $\lambda \in \a_Q \otimes_{\R} \C$ with $Re(\lambda) >> 0$.
 \begin{equation}
     C^{\mu}_{Q_1|Q_2}(\gamma_2p'sp^{-1}\gamma_1^{-1}, Ad(\gamma_1 p)(\lambda)) = f_{Q'}[\sigma_{\gamma_2}C^{\mu}_{Q|Q'}(s,\lambda)\sigma_{\gamma_1}^{-1}]f_Q
 \end{equation}
 where $(\sigma_{\gamma_i }\psi)(x) = \psi(x\gamma_i)$,
 $f_{Q'} = e^{(s\lambda +\tau_{Q'})(H_{Q'}(\gamma_2))}$ and $f_Q = e^{-(\lambda+\tau_Q)(H_Q(\gamma_1))}$.
\end{prop}
For a proof refer to \cite[Lemma 102]{HarishChandra1968}.

We still need to state and discuss a few more result before we can state the factorization properties of scattering matrices associated to non-minimal parabolic subgroups.

 We will again work with a pair of associate rational parabolic subgroup $\boldsymbol{Q_1},\boldsymbol{Q_2}$ of $\boldsymbol{G}$ with the real locus being $Q_1,Q_2$ and preassigned split components $A_{Q_1},A_{Q_2}$ respectively. Choose another parabolic subgroup $Q$ along with split component $A_Q$, such that $Q$ dominates both the parabolic subgroups $Q_1,Q_2$. Let $Q$ has an associated rational Langlands decomposition given by $Q = M_Q \times A_Q \times N_Q$. 
 
 Denote by $\tilde{Q_i}$ the rational parabolic subgroup in $M_Q$, defined as $\tilde{Q}_i = M_Q \cap Q_i$ along with rational split component $A_{i} = M_Q \cap A_{Q_i}$. For $s \in S(Q_1,Q_2)$  denote by $\tilde{s}$ the restriction of $s$ to $\a_1  =Lie(A_1)$, similarly for $\lambda_1 \in \a_{Q_1}^* \otimes_{\R}\C$, then call its restriction to $\a_Q^* \times_{\R}\C$ by $\tilde{\lambda}$. 
 
 \begin{prop}\label{prop:4_6}
 Let $Q_1,Q_2$ be as before let $s \in S(Q_1,Q_2)$ such that 
 $s|_{\a_{Q_1}} \equiv Id$, then for $\lambda_1 \in \a_{Q_1}^* \otimes_{\R}\C $ with $Re(\lambda_1) >> 0$ ,we have
 
 \begin{equation}
     C^{\mu}_{Q_1|Q_2}(s,\lambda)  =C^{\mu}_{\tilde{Q_1}|\tilde{Q_2}}(\tilde{s},\tilde{\lambda})
 \end{equation}
 \end{prop}
 For a proof refer to \cite[Lemma 108]{HarishChandra1968}.
 In order to obtain the full spectral resolution of $\Delta$ acting on 
$L^2 (\Gamma \backslash X)$, we need to set up some more notations. Note that the following discussion is a summary of the work done in section 13 of \cite{Ji2002} and is described here only for the purpose of giving a rather complete picture of the spectral theory of Locally symmetric spaces.

\pagebreak
For any $g \in L^2(i\a_Q^*)$, we introduce a function $\tilde{g} \in L^2(\Gamma \backslash X)$ as follows,
 
 \begin{equation}
     \tilde{g}(x) = \int_{i\a_Q^*} g(\lambda)E_Q(x,\psi,\lambda)d\lambda
 \end{equation}
 where $\psi \in L^2(X_Q)$ is an $L^2$ eigenfunction for the Laplacian on $X_Q$. For a fixed $Q,\psi$ denote by $L^2_{Q,\psi}$ to be the subspace spanned by functions of the form $\tilde{g}$ as above.

Let $\T$ be a set of associate rational parabolic subgroups of $G$. Let
$Q \in \T$, with an associated rational Langlands decomposition given by
$Q =MAN$.
Note that there is an set of  positive roots $\Sigma_{\T,+}$ 
such that the Lie algebra $n = Lie(N)$ admits a root space decomposition corresponding to the adjoint action of $A$ on $n$, i.e.

\begin{equation}
    n = \displaystyle \sum_{\beta \in \Sigma_{\T,+}} n^{\beta}
\end{equation}

Associated to $\Sigma_{\T,+}$,  $\a$ admits  Weyl chambers 
given by $C_1,C_2,...,C_t$. Note that associated to each of these chambers there is a rational parabolic subgroup $Q_{C_i}$ with rational Langlands decomposition $Q_{i} =MAN_i$, where
$n_i = \oplus_{\beta \in \Sigma_i}n^{\beta}$, with $n_i  = Lie(n_i)$ 

and 
$\Sigma_{i,+} = \{\beta \in \Sigma_{\T,+} | \beta(C_i)>0\}$. Let $\T_i =\{zQ_iz^{-1}| z \in G_{\Q}\}$,
 For each i let $j_i$ be the number of $\Gamma$ conjugacy classes in $\T_i$, with representatives $Q_{i \beta}$, with $1 \leq \beta \leq j_i$. For $\mu > 0$, let $\M_{i \beta}(\mu)$ be the $L^2$ eigenspace with eigenvalue $\mu$ of $S_{i\beta}$ (the boundary locally symmetric space associated to the rational parabolic subgroup $P_{i \beta}$). Note that the spaces are $\M_{i \beta}$ are non-empty only if $\mu$ is an actual eigenvalue. We will assume that 
$\mu \in \cup C_{i \beta}$, where $C_{i \beta}$ is the set of $L^2$ eigenvalues for the space $S_{i \beta}$.

Define the spaces,

\begin{equation*}
    \M_i(\mu) = \bigoplus _{\beta = 1}^{j_i}\M_{i \beta}(\mu)
\end{equation*}
To make things easier, we will assume that the Lie algebra $\a_{i\beta}$ to be $\a_i$, after conjugating by an appropriate element of $K$. More precisely, $\exists z_{i\beta} \in G_{\Q}$ such that
$Ad(z_{ik})(\a_i) = \a_{i \beta}$.

Choose $\psi_i  = (\psi_{i1},\psi_{i2},...,\psi_{ij_i})\in \M_{i}(\mu)$ and $\lambda _i \in \a_i \otimes_{\R} \C $, and define

\begin{equation*}
    E_{Q_i}(x,\lambda_i,\psi_i) = \sum_{\beta =1}^{j_i} E_{Q_{i \beta}}(x,\lambda_i, \psi_{i \beta})
\end{equation*}
Then for $w \in \mathcal{W}(\a_i,\a_j)$ we have a set of intertwining operators, $C^{\mu}_{ij}(w,\lambda_i): \M_i(\mu) \longrightarrow \M_j(\mu)$. Defined by,

\begin{equation}
    (\psi_j , C^{\mu}_{ij}(w,\lambda_i)(\phi_i))_{\M_i} = (\psi_j , C^{\mu}_{P_{i\beta_1}|P_{j \beta_2}}(z_{i\beta1}wz_{j\beta_2},Ad(y_{i\beta_1}(\lambda_i))), \phi_i)_{\M_{i\beta}}
\end{equation}
 where $\phi_i \in \M_{i \beta_1}(\mu)$ and
 $\psi_j \in \M_{j \beta_2}(\mu)$.

 Then the main result of Langlands is,
 
 \begin{prop}\label{prop:4_11}
 For any $k=1,2,...,r$ , $w_{ki} \in \mathcal{W}(\a_k,\a_i)$, $w_{kj} \in \mathcal{W}(\a_k,\a_j) $ satisfying $w_{ki} = w_{ji}w_{kj}$ we have,
 
 \begin{equation}
     C^{\mu}_{ki}(w_{ki},\lambda_i) = C^{\mu}_{kj}(w_{kj},w_{ji}\lambda_i)C^{\mu}_{ji}(w_{ji},\lambda_i)
 \end{equation}
 
 As well as the fact that $C^{\mu}_{ji}(w_{ji},\lambda_i)C^{\mu}_{ij}(w_{ji}^{-1},w_{ji}\lambda_i)=Id$, which implies if $\lambda_i$ is purely imaginary $C^{\mu}_{ki}(w_{ki},\lambda_i)$ is bijective and unitary.

 \end{prop}
 
 For a proof refer to Proposition 2.4.2 in \cite{Ji1999}.
 
 An immediate conclusion of this theorem is that
 the spaces $\M_i(\mu)$ all have the same dimension, denote this by D. Choose a set of basis $\Psi_{i}^1,\psi_i^2,...,\psi_i^D$ of the space $\M_i(\mu)$, such that each $\Psi_{i}^k$ belongs to one of the subspaces $\M_{i \beta}(\mu)$.
 
 For a given $g_i  = (g^1_i,g^2_i,...,g^D_i)$, where
 $g^k_i \in L^2(i\a_i ^*)$, we define a transform,
 
 \begin{equation}
     \tilde{f}^k_i(z) = \int_{i\a_i^*}g^k_i(\lambda_i)E_{Q_{i\beta}}(z,g_i^k , \lambda_i )d \lambda_i
 \end{equation}
where $\beta$ is uniquely determined by the fact that each $\Psi_{i}^k$ belongs to one of the subspaces $\M_{i \beta}(\mu)$.

and denote by $\tilde{f}_i(x)$ by the sum 
$\displaystyle \sum_{i=1}^D \tilde{f}^k_i(x)$.

Denoting by, $L^2_{\T,\mu}$ the subspace of 
$\prod_{i=1}^r \prod_{k=1}^D L^2(i \a_i^*)$, which consists of tuples of functions, $g =(g_1,g_2,...,g_D) =(g_1^1,...,g^D_1,g^1_2,...,g^D_2,...,g^1_r,...,g^D_r)$, that satisfy the condition:

for any $\{j_1,j_2\} \subset \{1,2,...,r\}$ along with $w_{j_1j_2} \in \mathcal{W}(\a_{j_1},\a_{j_2})$
we must have,

\begin{equation*}
    C^{\mu}_{j_1j_2}(w_{j_1j_2},\lambda_{j_2}) \sum_{k=1}^Dg^k_{j_2}(\lambda_{j_2})\psi^k_{j_2}=\sum_{k=1}^Dg^k_{j_1}(w_{j_1j_2}\lambda_{j_2})\psi^k_{j_1}
\end{equation*}

For any $g \in L^2_{T,\mu}$, define an associated norm function,

\begin{equation}
    ||g||^2 = \frac{(2\pi)^{-l}}{c}\sum_{j_1=1}^r\sum_{k=1}^D \int_{i \a_{j_!} ^*} |g^k_{j_1}(\lambda_{j_1})|^2 d\lambda_{j_1}
\end{equation}

where $l$ is the dimension of $\a_i$ and c is the number chambers on $\a_i$. Finally, define
$\tilde{g}(x) = \sum_{j_1=1}^r \tilde{g}_{j_1}(x)$.

The assignment $L^2_{\T,\mu} \ni f \mapsto \tilde{f} \in L^2(\Gamma \backslash X)$ is an isometric embedding into $\displaystyle\sum_{i=1}^r\sum_{\beta=1}^{j_i}\sum_{\psi \in \M_{i\beta}(\mu)}L^2_{Q_{i \beta},\psi}$. Call this image to be $L^2_{\T,\mu}(\Gamma \backslash X)$.

\begin{prop}\label{prop:4_15}

Further, for any rational parabolic subgroup $Q$, and an $L^2$ eigenfunction $\psi $ on the space $X_Q$, we identify $g \in L^2(i\a_Q^{*+})$ with a subspace of $L^2(\Gamma \backslash X)$ through the map,

\begin{equation*}
    L^2_{\T,\mu} \ni g \mapsto \tilde{g} = \int_{i\a_Q^{*+}}g(\lambda)E_Q(x,\psi,\lambda)d\lambda \in L^2(\Gamma \backslash X)
\end{equation*}
Then the continuous subspace $L^2_C(\Gamma \backslash X)$ is equal to the direct sum :

\begin{equation}
    \sum_Q \sum_{\psi} \bigoplus L^2(i\a_Q^{*+})
\end{equation}
where $Q$ runs over all the $\Gamma$-conjugacy
 classes of proper rational parabolic subgroups, and $\psi$ is over an orthonormal basis of $L^2$ eigenfunctions associated to the discrete spectrum of the space $X_Q$.
 \end{prop}

\begin{proof}
Refer to Proposition 13.14 in \cite{Ji2002}.
\end{proof}

\section{Scattering geodesics in Locally symmetric spaces}
\label{scatteringgeodesics}

Here we will review the work of Ji and Zworski on scattering geodesics from \cite{Ji2001}.
To make sense of the following discussion we assume that the
$\Q$-rank of $\Gamma \backslash X$ is one for the rest of the section. Recall that the $\Q$-rank of $\Gamma \backslash X$ is the maximal possible $\Q$-rank of a rational parabolic subgroup of $G$.

Let $\sigma : \R \longrightarrow M$ be a geodesic, where M is a Riemannian manifold. Then this geodesic is called 
$\textbf{Eventually Distance Minimizing}$ if $\exists t_0,t_1 \in R$ such that $\sigma|_{(-\infty,t_0]}$ as well as $\sigma|_{[t_1,\infty)}$ are both isometries onto their respective images in $M$. 

Let's recall some notations. If $Q$ is a rational parabolic 
subgroup of $G$, then associated to the Rational Langlands
decomposition of $Q$, there is a horospherical decomposition of the symmetric space $X$, described as follows.

Every $x \in X$, can be represented as $x  =(z_x ,exp_Q(H_x),n_x) \in X^Q \times A_Q \times N_Q$, with $H_x \in \a_Q$, where $\a_Q^+$ is the associated positive roots with respect to the set of positive roots
$\Sigma(Q,A_Q)$. Then for fixed $(z_x,n_x) \ in X^Q \times N_Q$ and $H \in \a_Q^+$, denote by $\tilde{\sigma} = (z_x,exp_Q(tH),n_x)$ the associated geodesic in $X$. We could choose $H$ with norm one (with respect to the Killing form), then $\tilde{\sigma}$ will be a unit speed geodesic.

Denoting by $\pi: X \longrightarrow \Gamma \backslash X$, the natural projection map, it can be shown easily that $\sigma  = \pi(\tilde{\sigma})$ is a eventually distance minimizing geodesic in $\Gamma \backslash X$.

The next theorem guarantees that the only possible eventually distance minimizing geodesics in $\Gamma \backslash X$ are of the above form.

\begin{prop}\label{prop:5_1}\cite[Theorem 10.18]{Ji2002}
Let $\sigma$ be an eventually distance minimizing geodesic 
$\Gamma \backslash X$, then $\sigma =\pi(\tilde{\sigma}) $, where $\tilde{\sigma}$ is a geodesic in $X$, of the form
$\tilde{\sigma} = (z,exp_Q(tH),n) \in X^Q \times A_Q \times N_Q$, with $(z,n) \in X^Q \times N_Q$ and $H \in \a_Q^+$ for some rational parabolic subgroup $Q$.
\end{prop}

\begin{defn}
A scattering geodesic in $\Gamma \backslash X$ is essentially an eventually distance minimizing geodesic in the sense of the above theorem for some rational parabolic subgroup $Q$
of $G$.
\end{defn}

\begin{prop}\cite[Proposition 2.6]{Ji2001}
A geodesic $\sigma(t) $ in $\Gamma \backslash X$ is a scattering geodesic from a parabolic subgroup $Q_1$ to another parabolic subgroup $Q_2$ if and only if for every $s >>0$, there exist $t_1,t_2 \in \R$ such that $\sigma(t_1) \in Y_{Q_1,s}$ and $\sigma'(t_1) \in N^+Y_{Q_1,s}$ and $\sigma(t_2) \in Y_{Q_2,s}$ with $\sigma'(t_2) \in N^-Y_{Q_2,s}$
Where, for a rational parabolic subgroup $Q$ of $G$, $Y_{Q,s}$ is the section at height $s >>0$ of the end associated to the rational parabolic subgroup $Q$, it is a codimension one submanifold of $\Gamma \backslash X$($dim(\a_Q) =1$), and $N^{+}Y_{Q,s}$ denote the connected component of the normal bundle $NY_{Q,s} \backslash \{0\}$ 
which contains the positive chamber $A_Q^+$ and and $N^{-}Y_{Q,s}$ as the other component.
\end{prop}

\begin{defn}
Let $s >> 0$ be a sufficiently large height, as in above theorem. For a scattering geodesic $\sigma(t)$ as in, let $t_1(s)$ be the largest number such that $\sigma(t)\in Y_{Q_1,s}$, the first time as $s$ decreases from $\infty$. Similarly, let $t_2(s)$ be the smallest number such that $\sigma(t)\in Y_{Q_2,s}$, the first time as $s$ increases from $-\infty$. Then the associated \textbf{sojourn Time} $T_{\sigma} = t_2(s)-t_1(s) -2s$, where the shift $2s$
is a normalization factor.
\end{defn}

The following theorems describes a way to parametrize the set of scattering geodesics between running between two end(which may be the same). For the proofs refer to section 2 in JZ.

\begin{prop}\label{prop:5_5}\cite[Proposition 2.9]{Ji2001}
Assume that the $\Q$-rank of $\Gamma \backslash X$ is equal to one. Let $\sigma(t)$ be a scattering geodesic in $\Gamma \backslash X $between the ends associated with two rational parabolic subgroups $Q_1$and $Q_2$.Then $\sigma(t)$ lies in a smooth family of scattering geodesics of the same sojourn time parametrized by a common finite covering space $X_{12}$ of the boundary locally symmetric spaces $X^{Q_1}$ and $X^{Q_2}$.
\end{prop}

\begin{prop}\label{prop:5_6}\cite[Proposition 2.10]{Ji2001}
Assume that the $\Q$-rank of $\Gamma \backslash X$ is equal to one as above. Then for any two(not necessarily different)ends of $\Gamma \backslash X$, there are countably infinitely many smooth families of scattering geodesics between them, and the spectrum of sojourn times of all scattering geodesics forms a discrete sequence of points in $\R$ of finite multiplicities.
\end{prop}

\pagebreak

Finally, the following theorem describes a relationship between sojourn times of scattering geodesics and geodesic flow on normal bundle associated to sections of ends associated to various parabolic subgroups.

\begin{prop}\cite[Proposition 2.11]{Ji2001}
Assume that the $\Q$-rank of $\Gamma \backslash X$ is equal to one as above. Let $Q_1,Q_2,...,Q_n$ denote a set of representatives of $\Gamma$- conjugacy classes of rational parabolic subgroups of $G$. Let $\Psi^t$ be the geodesic flow in the tangent bundle $T(\Gamma \backslash X)$ minus the zero section. For $s >> 0$, let $Y_{Q_i,s}$ denote the section at height s of the end associated to $Q_i$. 

Then $\Psi^t(NY_{Q_i,s}) \cap NY_{Q_j,s}$ iff one of the following two conditions is satisfied.

\begin{itemize}
    \item[$\bullet$] $t = 0$, in which case $Q_i = Q_j$ along with $Y_{Q_i,s} = Y_{Q_j,s}$. 
    
    \item[$\bullet$] $|t| \neq 0$, in which case $|t|-2s$ is the sojourn time associated to a family of scattering geodesic associated to ends corresponding to $Q_i$ and $Q_j$, and each connected component of this family is parametrized by a common finite covering space $X_{12}$ of the boundary locally symmetric spaces $X^{Q_1}$ and $X^{Q_2}$.
\end{itemize}
\end{prop}

\begin{prop}\label{prop:5_8}
If $Q_i$ and $Q_j$ are distinct $\Gamma$-conjugacy classes representatives of rational parabolic subgroups of G. Then for $\gamma \in \Gamma$, the sojourn time of the family of scattering geodesics between $Q_i$ and $\gamma Q_j \gamma^{-1}$ is given by $||log(A)||$ (norm with respect to the Killing form on $A \in A_Q)$, where $\gamma = v_2Amwv_1$ is its associated Bruhat decomposition, where $w \in K$ is a representative in the normalizer of $A_{Q_i}$ of the non-trivial element in the associated Weyl group.
\end{prop}

\begin{prop}
For every pair of (not necessarily distinct) ends of
 $\Gamma \backslash X$, the set of sojourn times associated to families of scattering geodesics is discrete with finite multiplicities.
\end{prop}

We are now ready to state the main result from \cite{Ji2001}.

Let $\Gamma \backslash X$ be a locally symmetric space of $\Q$ rank one. Let $R_1,R_2,...,R_d$ be a set of representatives of $\Gamma$ conjugacy classes of rational parabolic subgroups of $G$. Denote by 
$C_i$ the set of $L^2$ eigenvalues associated to the Laplacian for the boundary locally symmetric space $S_{R_i}$. Then for $\eta \in \bigcup_{i=1}^d C_i$, define the $\mu$-eigenspace $E_{\mu}$ by the direct sum
$\oplus _{j=1}^d \mu_{E}((L^2(S_{R_j})))$, where $\mu(E)(L^2(X))$ is the eigenspace with eigenvalue $\mu$. Choose an orthonormal set of basis $\psi_1,\psi_2,...,\psi_D$ for $E_{\mu}$ with
$D = dim(E_\mu)$ such that each $\psi_n$ belongs to a unique $\mu_E(L^2(S_{R_{F(n)}}))$, where $1 \leq F(n) \leq D$. Denote by $S^{\mu}_{n_1n_2}(\tau)$, the component of the scattering matrix for the eigenvalue $\mu$ with respect to the orthonormal basis described above.
Let $\mathcal{T}_{m_1 m_2}$ denote the set of sojourn times between the ends of $Q_{m_1}$ and $Q_{m_2}$, where $m_i = F(n_i)$.

\begin{prop}
    sing supp $\mathcal{F}(S^{\mu}_{n_1n_2}(\tau)) \subset \mathcal{T}_{m_1m_2}$, where $\mathcal{F}$ denotes the Fourier transform. Denoting by $J = \mathcal{F}(S^{\mu}_{n_1n_2}(\tau))$, we also have
    \begin{equation*}
        \scalebox{0.85}[1]{$J = \sum_{T \in \mathcal{T}_{m_1m_2}} e^{-\rho T} \frac{(2 \pi)^{1-(1/2)e_{T,m_1m_2}} }{q_{m_1} q_{m_2}}  \big ( \int_{X^T_{m_1m_2}}\pi^*_{m_1,T}\phi_i\pi^*_{m_2,T}\phi_j)
        \sum_{\pm} (\tau -T \pm i0)^{-1/2(1+e_{T,m_1m_2})}(1+g^T_{m_1m_2}(\tau))$}
    \end{equation*}
    where $e_{T,m_1m_2}$ is the dimension of the parametrizing space for the family of scattering geodesics between the ends of $Q_{m_1}$ and $Q_{m_2}$.
\end{prop}
For the proof and more details about some of the notations refer to Theorem 2 in \cite{Ji2001}.

\section{Scattering flats in locally symmetric spaces}

To make sense of this discussion we will assume that the $rank_{\Q}{\boldsymbol{G}} = rank_{\R}(G)$ for the rest of this Chapter, (for example think of $G$ as $SL(n,\R)$).

We will introduce higher dimensional analogues of scattering geodesics in this section and study their properties. We start by studying certain embedded submanifold of $X$, note that since $X$ is a symmetric space it can be shown that there is a one--one correspondence between flat totally geodesic submanifolds of $X$, and abelian subalgebras of $\gg$.

Let $Q$ be a rational parabolic subgroup $G$, with rational Langlands decomposition given by $Q=M_Q \times A_Q \times  N_Q$, and let $\Sigma^+(Q,A_Q)$ be the set of positive roots corresponding to the adjoint action of $A_Q$ on $n_Q = Lie(N_Q)$. Let $\a_Q^+$ be the positive chamber in $\a_Q$ with respect to the above root system, choose 
$H_1,H_2,...H_q \in \a_Q^+$ such that, where $H_1,H_2,...,H_q$ are linearly independent with $1\leq q \leq Rank(Q)$. Now consider the totally geodesic flat submanifold $\tilde{S}$ given by the set of points $(x,exp_Q(t_1H_1+...+t_qH_q),n)$ with $(x,n) \in (X^Q \times N_Q)$ and $t_i \in \R$. Note that this defines an isometric embedding of the Lie algebra $\a_Q$ into $X$, since the restriction of the metric from $X$ on $\tilde{S}$
is precisely the Killing form on $\a_Q$.

Denoting by the canonical projection $\pi : X \longrightarrow \Gamma \backslash X$, note that $S = \pi(\tilde{S})$ is a flat embedded submanifold of $\Gamma \backslash X$. Then note that for $r>>0$, with $T_r  = rT$ as before from \autoref{prop:3_5}, then it can be clearly seen that $S$ is not contained in the compact core $X^T_r$ for $r>>0$, this is because denoting by the geodesic $\tilde{\sigma} = (x,exp_Q(t_1H_1),n) \subset S$, then by \autoref{prop:5_1} the geodesic $\sigma  = \tilde{\sigma} \subset S$ in $\Gamma \backslash X$ is eventually distance-minimizing and hence runs of to infinity in both direction and is not contained in the compact core $X^T_r$.

Further note that, owning to the precise reduction theory, one can choose $r>>0$ and $T_r = rT$, then thinking of the flat $S$, as a map from $\a_Q$ to $\Gamma \backslash X$, note that the restriction of this map to the shifted Weyl chamber $\a_Q^+(T_r)$ is an isometric embedding into $\Gamma \backslash X$.

We need to discuss in more detail the embedding of $\pi(\tilde{S})$ into $\Gamma \backslash X$ which we describe now.

Let $\tilde{x}_0 \in X$, we define the Dirichlet fundamental domain associated $\tilde{x}_0$ with respected to the action of $\Gamma$ on $X$, by

\begin{equation}
    D_{\Gamma}(\tilde{x}_0) = \{x \in X | d_X(x,\tilde{x}_0) \leq d_X(x,\sigma \tilde{x}_0) | \forall \sigma \in \Gamma\}
\end{equation}
Now denote by $x_0 = \pi(\tilde{x}_0)$, and let $\gamma$ be a geodesic $\Gamma \backslash X$ such that $\gamma(0) = x_0$ along with its lift $\tilde{\gamma}$ to $X$ satisfying $\tilde{\gamma}(0) = \tilde{x}_0$, then it is a well-known result that $\gamma$ is eventually distance minimizing in $\Gamma \backslash X$ if $\tilde{\gamma}(t) \in D_{\Gamma}(\tilde{x}_0)$ for $t >>0$.

\pagebreak

A similar result holds for the scattering flat describe before which we will now discuss, let $\tilde{x}_0 = (x,exp_Q(t_1^0H_1+...+t_q^0H_q),n) \in S$, then for $t_i >> t_i^0$ 

we have,
$\tilde{S} \ni x =  (x,exp_{Q}(t_1H_1+...+t_qH_q),n) \in D_{\Gamma}(\tilde{x}_0)$ and such that we have the relation, $d_{\Gamma \backslash X}(x_0,\pi(x)) = d_{\a_Q}(t_1^0H_1+...+t_q^0H_q, t_1H_1+...+t_qH_q)$

This means, thinking of the flat $S$ as an embedding map from $\a_Q$ to $\Gamma\backslash X$, we have that the restriction of this map to $\a_Q^+(T,r)$ is an isometric embedding into $\Gamma \backslash X$ for $r >> 0$.
\begin{defn}\label{def:5_3}
Let $\Gamma \backslash X$ be a locally symmetric space as before, then a flat embedded submanifold $S$ of $\Gamma \backslash X$ is called a scattering flat if it is of the form $\pi(\tilde{S})$ where $S$ is a totally geodesic flat submanifold of $X$ associated to a parabolic subgroup $Q$ as before.
\end{defn}

\begin{defn}\label{def:6_3}
Let $\a$ be an abelian subalgebra of $\gg$, denote by $\Sigma$ the set of roots corresponding to the adjoint action of $\a$ on $\g$, then a scattering flat is a flat submanifold $S$ of $\Gamma \backslash X$ of dimension equal to rank$(\a)$, given by a smooth embedding $\Psi: \a \longrightarrow \Gamma \backslash X$, such that for any choice of a full subset of positive roots $\Sigma^+$ of $\Sigma$ wit the associated positive chamber $\a^+$, we have that the restriction of $\Psi$ to the shifted Weyl chamber $\a^+(\Sigma, H)$ is an isometric embedding onto $\Gamma \backslash X$, where $H \in \a^+$ with $|H| >> 0$ and
$\a^+(\Sigma,H) = \{X \in \a | \beta(X-H) > 0 \forall \beta \in \Sigma^+\} $
\end{defn}

In light of the discussion earlier, we see that the two definitions stated previously are in fact equivalent in the sense that, any scattering flat in $\Gamma \backslash X$ is of the form $S  = \pi(\tilde{S})$, where $S = (z,exp_Q(t_1H_1+...t_qH_q),n), t_i \in \R$ is a flat submanifold of $X$, where $q = rank(Q)$ with $Q$ being a rational parabolic subgroup 
with the associated horospherical decomposition of $X$ given by the product $X^Q \times A_Q \times N_Q$, $H_1,...,H_q \in \a_Q^+$ are linearly independent, where $\a_Q^+$ is the positive Weyl chamber associated a set of positive roots $\Sigma^+(Q,A_Q)$ corresponding to the adjoint action of $\a_Q$ on $n_Q = Lie(N_Q)$, $(z,n) \in X^Q \times N_Q$.

Before we proceed further, let's recall the precise reuction theory of $\Gamma \backslash X$ from \autoref{prop:3_5}. It essentially says that $\Gamma \backslash X$ can be decomposed as a disjoint union of a compact core and a finite number of Siegel ends associated to a set of representatives of rational parabolic subgroups corresponding to the $\Gamma$-conjugacy classes in $G$.

Choose a scattering flat in $\Gamma \backslash X$ of the form $S = \pi(\tilde{S})$ as before, choose an associate non-$\Gamma$-conjugate rational parabolic subgroup $Q_1$ with the same split component as that of $Q$(For the moment choose $Q_1$ to be different from $Q$), so $Q_1$ has a rational Langlands decomposition given by $Q_1=M_Q\times A_Q \times N_{Q_1}$ with $N_Q \neq N_{Q_1}$ and

\begin{equation*}
    n_1 = Lie(N_{Q_1})  = \bigoplus_{\beta_1 \in \Sigma^{+}_1}n_Q^{\beta_1}
\end{equation*}
 with $\Sigma^{+}_1$ being the set of positive roots for the adjoint action of $A_Q$ on $n_{1}$. 
 
\pagebreak
 
 Denote by $\a_1^+$ the associated positive chamber, then since the Weyl group works transitively on Weyl chambers, we can choose elements $Z_1,...,Z_q \in \a_1^+$ corresponding to the set $\{H_1,H_2,...,H_q\}$. Finally, from (add reference) about precise reduction theory, we can (after a possible reordering), make sure that $R_1$ is $\Gamma$ conjugate to $Q$ and $R_2$ is $\Gamma$ conjugate to $Q_2$, the from the discussions we have been having, we see that the scattering flat $S$ runs between the Siegel ends of the form $\Omega \times A^+_{R}(T_r)$ showing up in (add reference), corresponding to the two parabolic subgroups $R =R_1,R_2$.

 Note that fixing $Q$ as before but replacing $Q_1$ by $\gamma Q_1 \gamma^{-1}$, $\gamma \in \Gamma$, we can produce a different scattering flat running between the same Siegel end.

Note that the construction described before, can be modified slightly to give scattering flats between same Siegel end, to do this we proceed as follows, choose
$\gamma \in \Gamma, \gamma \notin Q$, and let $Q_1 = \gamma Q \gamma^{-1}$, then $Q,Q_1$ are different but are $\Gamma$ conjugate, so correspond to the same Siegel end from the reduction theory, and the construction outlined before works fine.

Also, observe that the procedure outlined before gives all the possible scattering flats running between the Siegel ends corresponding to $Q$ and $Q_1$, as well as shows that \textbf{the necessary and sufficient condition for the existence of scattering flats, which is that we just need $Q$, $Q_1$ to be associate. This will be important later because the scattering matrices associated to two non-associate rational parabolic subgroups always vanish identically}.

Finally, let's make an observation that unlike a scattering geodesic, a higher dimensional scattering flat can scatter between multiple pairs of Siegel ends.

We work as before $Q,Q_1$ be a pair of associate rational parabolic subgroup with a common split component as before, the next thing we want to do is come up with a parametrization space for scattering flats running between the associated Siegel ends. We proceed as follows,

\begin{itemize}
    \item[$\bullet$] Start with the scattering flat ${S} = \pi(\tilde{S})$, where $\tilde{S} = (z,exp_Q(t_1H_1+...+t_qH_q,Id))$, where $z \in X^Q$ and $Id$ is the identity element in $N_Q$. if $z$ is chosen to be a different point $X^Q$, then the resulting scattering flat in $\Gamma \backslash X$ will have the same \textbf{sojourn vector (An $\a_Q$ valued parameter that keep tracks of subset of the flat contained in the compact core $X_r$, it is independent of $r$, it will be defined in more detail in the next section}). Denote by $\Gamma^Q$ the image of $\Gamma \cap Q$ in $M_Q$, under the projection map $Q \longrightarrow M_Q$.

    Consider the two flats
    $\tilde{S}_i = (z_i,exp_Q(t_1H_1+...+t_qH_Q),Id)$ in $X$, where $z_i \in X^Q$, then $S_i  = \pi(\tilde{S}_i)$ represent the same scattering flat in $\Gamma \backslash X$ if and only if $\exists \gamma \in \Gamma | \gamma z_1  = z_2 \implies \gamma \in \Gamma \cap M_Q \implies \gamma \in \Gamma_Q \cap M_Q$, where $\Gamma_Q  = \Gamma \cap Q$. This means the scattering flat $S$ belongs to a smooth family parametrized by $(\Gamma_Q \cap M_Q) \backslash X^Q$ with a common sojourn vector.
    Let's remark that since $\Gamma_Q \cap M_Q$ is a finite index subgroup in $\Gamma^Q$, and so $(\Gamma_Q \cap M_Q) \backslash X^Q$ is a finite covering space of $S_Q  = \Gamma ^Q \backslash X^Q$.
    Also, note that since $M_{Q_1} = M_Q$ we have,
    $\Gamma_Q \cap M_Q  = \Gamma \cap M_{Q_1} = \Gamma_{Q_1} \cap M_{Q_1}$ and hence the spaces $(\Gamma_Q \cap M_Q) \backslash X^Q$ and $(\Gamma_{Q_1} \cap M_{Q_1}) \backslash X^{Q_1}$ can be identified.

    \item[$\bullet$] Suppose we start with a flat submanifold $\tilde{S} = (z,exp_Q(t_1H_1+...+t_qH_Q),n) $ with $(z,n) \in X^Q \times N_Q$, then we choose a new base point $x_1 \in X$, given by $x_1 = nx_0$, where $x_0 =eK \in G/K =X$. With respect to this new base point we get the following decompositions.
    
    The associated Rational Langlands decomposition of $Q$ is given by
    $Q = nM_Qn^{-1} \times nA_Qn^{-1} \times N_Q$.
    Denoting by $K_1 = n(K\cap M_Q)n^{-1}$, we have that the associated Horospherical decomposition of $X$ is given by
    $X = (nM_Qn^{-1})/K_1 \times nA_Qn^{-1} \times N_Q$.  
    
    With respect to this horospherical decomposition, the flat $\tilde{S}$ is represented as

    $\tilde{S} = (nzn^{-1}K_1,nexp_Q(t_1H_1+...+t_qH_q)n^{-1},Id)$.
    
    Note that $S  = \pi(\tilde{S})$ is a scattering flat in $\Gamma \backslash X$ running between Siegel ends corresponding to the rational parabolic subgroup $Q$ and $\tilde{Q} = nQ_1n^{-1}$ with the common split component $nA_Qn^{-1}$. We will again try to analyze the parametrizing space of such Scattering flats just as we did before.
    We can choose the $X^Q$ component of $\tilde{S}$ to be arbitrary and the resulting scattering flat $S = \pi(\tilde{S})$ would have the same sojourn vector as before.
    We will once again deal with Horospherical decomposition of $X$ with respect to the split component $A_Q$, choose two such  flats  in $X$ as before, namely $\tilde{S}_i = (z_i,exp_Q(t_1H_1+...+t_qH_Q),n)$, where
    $z_i \in X^Q$, then the associated scattering flats $S_1,S_2$ are the same in $\Gamma \backslash X$ if and only if 
    $\exists \gamma \in \Gamma $ such that
    $\gamma n z_1  = n z_2 \implies n^{-1}\gamma n \in n^{-1}\Gamma n \cap M_Q$. This shows that the scattering flat $S$ belongs to a smooth family parametrized by the space $(n^{-1}\Gamma n \cap M_Q) \backslash X^Q$.

    Since $\Gamma$ is torsion free, so is the group $n^{-1}\Gamma n \cap M_Q$, and so
    $(n^{-1}\Gamma n \cap M_Q) \backslash X^Q$ is a smooth manifold and as before it can be shown that this is a finite covering space for the associated locally symmetric spaces $S_Q$ and $S_{Q_1}$. Summarizing all this we get the following theorem,
    
    \begin{thm}\label{theorem:6_4}
    Let $\Gamma \backslash X$ be a locally symmetric space of $\Q$ rank greater than one,  for two rational parabolic subgroups $Q_1$ and $Q_2$ of $G$; a scattering flat exists between the associated Siegel ends(which could be the same) if and only if $Q_1$ and $Q_2$ are associate. In case $Q_1$ and $Q_2$ are associate, the choice of a common split component of $Q_1$ and $Q_2$ gives rise to a smooth family of scattering flats between the corresponding Siegel ends parametrized by a common finite cover of $S_{Q_1}$ and
$S_{Q_2}$.
    \end{thm}

\end{itemize}

\pagebreak

\section{Factorization of Scattering matrices}

We will now review the factorization of scattering matrices coming from minimal rational parabolic subgroups of $G$. As a byproduct, we will come up with a construction that will be useful while discussing projections of scattering flats in the next section. We proceed as follows,

Let $Q_1$ be a minimal rational parabolic subgroup with rational rank equal to $\text{q}$ and Langlands decomposition $Q_1 = MAN_1$ along with a set of positive roots $\Sigma^{+,1}$ and a set of simple roots $\Sigma^{++,1}$ corresponding to the adjoint action of $\a = \text{Lie}(A)$ on
$n_1 = \text{Lie}(N_1)$ along with a choice of a positive chamber $C_1 \subset \a$. Further let $\Sigma$ denote the full set of roots for the adjoint action of $\a$ on $\gg = \text{Lie}(G)$ along with subspaces of $\gg$ given by 
$\gg^{\beta} = \{Y \in \gg | [H,Y] = \beta(H)Y \hspace{0.1cm}\forall H \in \a \}$ associated to every root $\beta \in \Sigma$.

Note that for any chamber $C$ in $\a$ we can associate a set of positive roots $\Sigma^+(C)$ by defining ,

\begin{equation}
    \Sigma^+(C) = \{\beta \in \Sigma | \beta(Y)>0 \forall Y \in C\}\ 
\end{equation}
Then note that $\Sigma^+(C_1) = \Sigma^{+,1}$. Now fix an $\alpha \in \Sigma^{++,1}$, then there exists a unique chamber $C_{\alpha}$ which is adjacent to $C$ and we have $-\alpha \in \Sigma^+(C_\alpha)$, denoting by $\Sigma^+(C_{\alpha})$ as simply $\Sigma^{+,2}$ we define,
\begin{equation}
    n_2 = \bigoplus_{\beta \in \Sigma^{+,2}}g^{\beta}
\end{equation}
Then defining $N_2  = exp(n_2)$ along with $Q_2 =MAN_2$, we notice that $Q_2$ is a minimal parabolic subgroup associated to $Q_1$ having a common split component $A$. 

\begin{defn}\label{def:7_3}
Two minimal parabolic subgroups are said to be adjacent if with respect to a common split component, they can be constructed as above.
\end{defn}

Since the two rational parabolic subgroups $Q_1$ and $Q_2$ are coming from adjacent chambers one can construct a nontrivial rational parabolic subgroup containing both $Q_1$ and $Q_2$, the construction is rather easy and works as follows,

Define, $\a_P = \{ H \in \a | \alpha(H) = 0 \}$, and
let $\tilde{\Sigma}$ consists of roots from $\Sigma^{+,1}$ which do not vanish identically on $\a_P$.

We now define,

\begin{equation}
    n_P  = \bigoplus_{\beta \in \tilde{\Sigma}}g^\beta   \hspace{0.1cm} \text{and}  \hspace{0.1cm} g_P = Z(\a_P) \oplus n_P
\end{equation}
Where $Z(\a_P)$ is the centralizer of $\a_P$ in $\gg$. 

\pagebreak

Now define $A_P = exp(\a_P)$ and $N_P = exp(n_P)$ and denote by $P$ the normalizer of $g_P$ in $G$ under the adjoint action. Then $P$ is a rational parabolic subgroup with split component $A_P$ such that $P$ contains both $Q_1$ and $Q_2$. For future reference we will write the Langlands decomposition of $P$ as $P = M_PA_PN_P$. Note that the rational rank of $M_P$ is one.

In order to describe the factorization, we need to set up more notation. For a fixed parabolic subgroup $Q$ with split component $A_Q$ and $z \in G_{\Q}$ define the parabolic subgroup and split component pair $(Q,A_Q)^z$ by defining $Q^z = zQz^{-1} $ and $A_Q^z  = zA_Qz^{-1}$

Choose $z_1,z_2 \in G_{\Q}$ and define parabolic subgroups, $(P_i,A_i) = (P,A_P)^{z_i}$ and 
$(Q^i,A^i) = (Q_i,A)^{z_i}$. Then by [Lemma 106, HC]
the scattering matrix $C_{Q^1|Q^2}(z_2z_1^{-1},Ad(z_1)(\lambda))$ vanishes identically unless $P_1$ and $P_2$ are $\Gamma$-conjugate, where $\lambda \in \a^* \otimes \C$. 

So, we can just focus on the case when $z_1 = \gamma$ is an element from $\Gamma$ and $z_2  = Id$. Note that in this case $P_2$ is just the previously constructed parabolic subgroup $P$ and $Q^2  =Q_2$. We can further choose a $u \in (N_P)_{\Q}$ such that $A_2^{\gamma u} = A_1$, in fact since $A_2  = A_P$ and $A_1  = A_P^{\gamma}$ we can simply choose $u=1$ and get $A_2^{\gamma}=A_1$.

Denote by $S(Q^1,Q^2)$ as the set of the maps of the form $Ad(y) $, for $y \in G_{\Q}$ such that
$Ad(y)A^1 = A^2$. Then choose $w \in S(Q^1,Q^2)$ given by $w = Ad(z_2z_1^{-1})$.
and set $\lambda_0 = Ad(\gamma)(\lambda)$.

Notice that our construction guarantees that the parabolic subgroup $(P_1,A_1)$ dominates both
$(Q^1,A^1)$ and $(Q^2,A^2)^{\gamma }$ in the sense that $Q^1,Ad(\gamma)(Q^2) \subset P_1 $ and $A_1 \subset A^1 , Ad(\gamma)(A^2)$. For simplification of notation denote by $(Q^3,A^3) = (Q^2,A^2)^{\gamma }$. Note that $Q^1,Q^3$ are associate sharing a common split component given by $A^1$ along with rational Langlands decomposition given by $Q^i = \tilde{M} A^1 N^i$, for $i =1,3$, where $\tilde{M} = \gamma M \gamma^{-1}$.

Finally, define parabolic subgroups $R_1,R_3$ of $M_{P_1}$ by $R_i = M_{P_1} \cap Q^i$, with $ i =1,3$.
Then the two parabolic subgroups are of rational rank one and have the Langlands decomposition given by, $R_i =  \tilde{M}\tilde{A}\tilde{N}^i$, with $\tilde{A} = M_{P_1} \cap A^1$ satisfying $A^1 = \tilde{A}A_{P_1}$ and $\a^1  = \tilde{a} \oplus \a_{P_1}$, with $\tilde{\a} = \text{Lie}(\tilde{A})$ and $\a_{P_1}  = \text{Lie}(A_{P_1})$. Note that the decomposition of $\a^1$ described here is in fact orthogonal with respect to the Killing form on $\a^1$. 

Choose $\lambda \in \a^* \otimes \C$ and define
$\lambda_0 =  Ad(\gamma)(\lambda)$. The splitting $\a^1  = \tilde{a} \oplus \a_{P_1}$, naturally induces a restriction of $\lambda_0$ onto $\tilde{a}$ given by $\tilde{\lambda}$. Also, let $\tilde{w}$ be the restriction of $w$ onto $\tilde{a}$. Then we have the following factorization.

\begin{equation}
    C_{Q^1|Q^3}(\gamma u w ,\lambda)  = C_{R_1|R_3}(\tilde{w} , \tilde{\lambda}_0) 
\end{equation}
For a proof refer to \cite[Lemma 116]{HarishChandra1968}.

\pagebreak
Note that an analogue of the factorization above is possible even if the parabolic subgroups $Q_1$ and $Q_2$ doesn't correspond to adjacent chambers, however in that case the right hand side of the above equation has to be replaced by a product of rank one scattering matrices. Once again, for details refer to Chapter 5 of \cite{HarishChandra1968}.

\section{Projection of Scattering flats and associated Scattering geodesics}

We will work with a pair of associate minimal rational parabolic subgroups of $G$, and the first thing we will show is that any two such associate rational parabolic subgroups have a common rational split component giving rise to a family of scattering flats with a common sojourn vector.

\begin{prop}\label{prop:8_1}
Any two associate distinct minimal rational parabolic subgroups of $G$ have a common split component.
\end{prop}

\begin{proof}
Start with a minimal rational parabolic subgroup $Q=MAN$ of $G$ and denote by $\a$ the Lie algebra of $A$. We will use the notation $Q^g = gQg^{-1}$ where $g \in G$. Then for any non-trivial $\gamma \in \Gamma, \gamma \notin Q$ the two parabolic subgroups $Q,Q^{\gamma}$, are distinct but $\Gamma$-conjugate hence correspond to the same Siegel set in the reduction theory of $\Gamma \backslash X$.

Start with the Bruhat decomposition 
$\gamma = u_2 \gamma_azwu_1$, where $u_1,u_2 \in N$, $\gamma_a \in A$ and $z \in M$, with $w \in K$, a representative in the normalizer of $A$ representing a non-trivial element of the Weyl group.

Denoting by $Q_1 = Q^{u_2^{-1}\gamma} $, then $Q_1 = Q^{\gamma_a z w u_1} = Q^{\gamma_a z w} \supset A^{\gamma_a z w} = A$, so $Q^{\gamma} \supset A^{u_2}$. Clearly, $Q \supset A^{u_2}$, hence $A^{u_2}$ is a common split component of $Q$ and $Q^{\gamma}$ and will give rise to a family of scattering flats in $\Gamma \backslash X$ scattering between the same Siegel set in the reduction of $\Gamma \backslash X$.

If $Q_1,Q_2$ are associate non-$\Gamma$ conjugate distinct minimal rational parabolic subgroups in $G$, then they correspond to different Siegel sets in the reduction theory of $\Gamma \backslash X$ in (add reference). Then for any $\gamma \in \Gamma$ the parabolic subgroups $Q_1^{\gamma},Q_2$ are still distinct, since $Q_1,Q_2$ are associate, we can choose a $k \in K$ such that $Q_1 ^k = Q_2$, Now consider the Bruhat decomposition of $k^{-1}\gamma  = u_2 \gamma_a  z w u_1$, with $u_1,u_2 \in N$, $\gamma_a \in A$ and $z \in M$, along with $w \in K$, a representative in the normalizer of $A$ representing a non-trivial element of the Weyl group.

Then the pair $(Q_1^{\gamma},Q_2)$ is conjugate to the pair $(Q_1, Q_1^{k^{-1}\gamma})$ and the proof above can again be used to produce a common split component of $Q_1,Q_2$ which in turn gives rise to a family of scattering flats in $\Gamma \backslash X$.
\end{proof}

\pagebreak
\begin{defn}\label{def:8_2}
Choose $Q$ as before, then $\text{log}(\gamma_a) \in \a$ is called the \textbf{sojourn vector} associated to the family of scattering flats.
\end{defn}

We can now summarize our discussion in the following result.

\begin{thm}\label{theorem:8_3}
Let $\Gamma \backslash X$ be a locally symmetric space of rational rank $q$. Fix two distinct minimal rational parabolic subgroups $\mathbf{Q}_1$ and $\mathbf{Q}_2$ of $\mathbf{G}$. Then, for any $\gamma \in \Gamma$ their exists a common rational split component  $A_{\gamma}$ of $Q_1$ and $Q_2$. Further suppose that with respect to the Langlands decomposition $Q_1 = MA_{\gamma}N$,
$\gamma$ admits a Bruhat decomposition given by $\gamma = u_2\gamma_azwu_1$ with $u_1,u_2 \in N$, $\gamma_a \in A_{\gamma}$, $z \in M $ and $w \in N_K(A_{\gamma})$. Then the common split component $A_{\gamma}$ gives rise to a family of $q$ dimensional scattering flats between the Siegel ends corresponding to $Q_1$ and $Q_2$ with a common sojourn vector which is defined to be
$log(\gamma_a) \in Lie(A_{\gamma})$.
\end{thm}

We will now associate a family of scattering geodesics to the family of scattering flats constructed before. This technique works under the assumption that the two associate parabolic subgroups correspond to adjacent chambers which we now elaborate.

Let $Q_1,Q_2$ be distinct associate minimal rational parabolic subgroups of $G$ of $\Q$-rank equal to $\text{q}$ and corresponding to adjacent Weyl chambers in the sense of \autoref{def:7_3}. 
Then as described in the previous section, for any
$\gamma \in \Gamma$ there is a rational parabolic subgroup $P$ which is of rational rank $\text{q}-1$
and contains both $Q_1$ and $Q_2^{\gamma}$.

Denote by $Q_3 = Q_2^{\gamma}$, then as before we know that $Q_1,Q_3$ have a common split component $A$ and we have the associated Langlands decomposition $Q_i = MAN_i$ for $i =1,3$.

Choose the Burhat decomposition for $\gamma $ given $\gamma = u_2 \gamma_azwu_1$, where $u_1,u_2 \in N$, $\gamma_a \in A$ and $z \in M$ along with $w \in K$, a representative in the normalizer of $A$ representing a non-trivial element of the Weyl group. Let $P = M_PN_PA_P$ be the associated Langlands decomposition, note that $M_P$ has $\Q$-rank one (when thought of as a reductive group in its own right), hence the associated Boundary locally symmetric space $S_P$ has rational rank one.

To proceed further define the rational parabolic subgroups $Q^i = M_P \cap Q_i$, then $Q^1$ , $Q^3$ are both rational parabolic subgroups of $M_P$ and hence are automatically associate(since $Rank_{\Q}(M_Q)$ = 1). They have corresponding Langlands decomposition given by $Q^i = M\tilde{A}N^i$, with $\tilde{A} = M_P \cap A$ satisfying $A = \tilde{A}A_P$ and $\a  = \tilde{a} \oplus \a_P$, with $\tilde{\a} = \text{Lie}(\tilde{A})$ and $\a_P  = \text{Lie}(A_P)$. Note that the decomposition of $\a$ described here is in fact orthogonal with respect to the Killing form on $\a$.

We also have $N_i = N^i N_Q$, denote by $\pi_i : N_i \longrightarrow N^i$ the natural projection maps. Note that $\tilde{\a}$ is a one-dimensional real vector space, we will be identifying it with $\R$ such that the restriction of the norm from the Killing form to $\tilde{\a}$, will just be the Euclidean norm on $\R$.

Finally, note that since $N^1 \neq N^3$ we can assume that $Q^1$ and $Q^3$ are opposite parabolic subgroups.

\pagebreak

Now we start with a scattering flat S in $\Gamma \backslash X$, which is the projection of a flat $\tilde{S}$ in $X$ of the form $S = (z,exp_{Q_1}(t_1H_1+...+t_qH_q),n_1)$, with $t_i \in \R$ and $H_i \in \a^+$ and $(z,n_1) \in X^{Q_1} \times N_1$ and as usual $\a^+ $ is the positive chamber in $\a$ corresponding to the choice of a set of positive roots for the adjoint action of $\a$ on $n_1 = \text{Lie}(N_1)$.

We will project this scattering flat S onto a scattering geodesic in the boundary locally symmetric space associated to the parabolic subgroup $P$ given by $S_P$, we now describe the construction.

We start with the horospherical decomposition of the symmetric space $X^P$ given by,

$X^P = X^{Q_1} \times \tilde{A} \times N^i$. Let $\tilde{\sigma}(t)$ be the geodesic in $X^P$ which is given in the horospherical decomposition just described, by $\tilde{\sigma}(t) = (z,exp_{Q^1}(tH),\pi_1(n_1))$, where $t\in \R$ and $H \in \tilde{\a}^+$.
Note that changing $H$ just changes the speed of the geodesic, but the geodesic as a curve remains the same.

Finally, let $\sigma(t)$ be the projection of $\sigma$
into $S_P$ under the canonical projection map from $X^P$ onto $S_P$. We define $\sigma$ to be the scattering geodesic in $S_P$ which will be called a projection of the scattering flat $S$ in $\Gamma \backslash X$.

We will finish this section by calculating the sojourn time of $\sigma$ in terms of the sojourn vector of the scattering flat S in $\Gamma \backslash X$.

Before we proceed, we need the following notion:

For a $\Q$-rank one parabolic subgroup $R$ with rational split component $A$ and Langlands
decomposition $R = MAN$, along with the unique positive root $\alpha$ corresponding to the
adjoint action of $a =\text{Lie}(A)$ on $n= \text{Lie}(N)$, define the horosphere at level $r \in \R^+$, by
$B_{P,r} = \{Y \in A|\alpha(\text{log}(Y)) = r\}$.

\begin{thm}\label{theorem:8_4}
Let $\sigma(t)$ be the scattering geodesic in $S_P$ constructed earlier, then the normalized sojourn time associated to $\sigma(t)$ is precisely $|\text{log}(\gamma_a)|$($\gamma_a$ as before), where $|\bullet|$ is the norm associated to the restriction of the Killing form on $\tilde{\a}$.
\end{thm}

\begin{proof}
Note that to calculate the sojourn time of $\sigma(t)$, is the same thing as calculating the sojourn time for $\tilde{\sigma}$(Refer to the discussion in section 6). Also note that the sojourn time will be independent of the choice of $(z,n_1)$, so we might as well work with the geodesic $\tilde{\sigma}(t) = (Id,exp(tH),Id)$. We can choose $H \in \tilde{\a}^+$ large enough and using the identification of $\tilde{\a}$ with $\R$ we can just choose $r =H$ so that this geodesic intersects both the horospheres $B_{Q_1,r}$ and $B_{Q_3,r}$. Then the sojourn time will just be the difference of the time instants corresponding to these two intersection points.

Note that $\sigma(t)$ intersect $B_{Q^1,r}$ at the point $\sigma(t_1) = (Id,exp_{Q_1}(H),Id)$ and it intersects $B_{Q^3,r}$ at $\sigma(t_2) = (Id, exp_{Q_1}(\text{log}(\gamma_a)-H),Id)$, so the distance between the two horospheres is $|log(\gamma_a)|+2r$ and consequently the normalized sojourn time is $|\text{log}(\gamma_a)|$.
\end{proof}

We summarize the results of this section in the following theorem.

\pagebreak

\begin{thm}\label{theorem:8_5}
Let $\Gamma \backslash X$ be a locally symmetric space of rational rank $q$. Fix two distinct minimal rational parabolic subgroups $Q_1$ and $Q_2$ of ${G}$ which correspond to adjacent chambers in the sense of the above definition. Then, for any $\gamma \in \Gamma$ there exists a rational parabolic subgroup ${P}$ of rank $q-1$ in ${G}$ which contains both $Q_1$ and $Q_2^{\gamma}$. Further the family of scattering flats constructed in \autoref{theorem:8_3} projects to a family of scattering geodesics in the $\Q$-rank one locally symmetric space $S_P$ with a common normalized sojourn time given by $|log(\gamma_a)|$, where $|\bullet|$ denote the norm induced by the killing form on $Lie(A_{\gamma})$.
\end{thm}

\section {Scattering on Quotients of \texorpdfstring{$\text{SL}(3,\R)$}{SLR}}

In this section we will study the spectral resolution and scattering flats for the locally symmetric space $ \Gamma \backslash X$ with $X = G/K$, where for the rest of the section we fix the following notation. $G $ 
will denote the semisimple lie group $SL(3,\R)$ with the associated eight dimensional Lie algebra $\gg$. $K$ and $\Gamma$ will stand for $SO(3)$ and $SL(3,\Z)$ respectively. The first task is to review the Lie algebra $\gg$ and its root space decomposition along with the construction of standard parabolic subalgebras of $\gg$.

Let $M(3,\R)$ be the vector space of 3 by 3 matrices with real entries, one can endow $M(3,\R)$ with a Lie algebra structure by simply defining the lie bracket of $X,Y \in M(3,\R)$ by defining $[X,Y] = XY-YX$ where for $X,Y \in M(3,\R)$, $XY$ represents the matrix multiplication. Then $\gg$ is the subalgebra of $M(3,\R)$ consisting of matrices of trace zero. The associated Killing form $B$ is given by the map $B: \gg \times \gg \longrightarrow \gg $, given by B(X,Y) =6Trace(XY).

Let $E_{ij}$ denote the three by three matrix all of whose entires are zero, except for the $ij$-th entry which is 1, then it can be easily shown that $\gg$ can be written as a direct sum(in the sense of vector spaces), as 
\begin{equation*}
 \gg = \displaystyle H \oplus_{i \neq j} \R E_{ij}
\end{equation*}

where $H$ is the Cartan subalgebra of $\gg$ generated as a two dimensional vector space by $H_1 = E_{11} -E_{22}$ and $H_2 = E_{22}-E_{33}$, note that as a set, $H$ is just the vector space of 3 by 3 real diagonal matrices with trace zero.
We will use the notation $h=(h_1,h_2,h_3)$ to denote an element of $H$ which is a diagonal matrix with diagonal elements $h_1,h_2,h_3$ such that $h_1+h_2+h_3 =0$. Note that the restriction of the Killing form $B$ to $H$ is positive definite and can be used to define an inner product in $H$. Finally, Let $H^*$ denote the real vector space dual of $H$, and we define the functionals $\alpha_{ij} : H \longrightarrow \R$, by $\alpha_{ij}(h) = h_i-h_j$, where $i,j = 1,2,3$.

Note that, the action of the Cartan subalgebra $H$ on $\g$ is simultaneously diagonalizable which leads to the root space decomposition of $\gg$, we have

\begin{equation}
\gg = H \oplus_{i \neq j} \gg^{ij}
\end{equation}
with $\alpha_{ij} \in H^*$ as above and $\gg^{ij } = \{Y \in \gg  \hspace{0.2cm}|  \hspace{0.2cm} [h,Y]   = \alpha_{ij}(h)Y  \hspace{0.2cm}\forall  \hspace{0.2cm} h \in H \}$.

Note that for the adjoint action of $H$ on $\gg$, $\Sigma^{++} =   \{\alpha_{12} , \alpha_{23} \}$ serves as a set of simple roots and the corresponding set of positive roots is given by $\Sigma^+ = \{\alpha_{12} , \alpha_{23} ,\alpha_{13}\}$ and the full set of roots is then $\Sigma = \{\pm\alpha_{12} , \pm\alpha_{23} ,\pm\alpha_{13}\}$.
We further define $\tau: H \longrightarrow \R$, to be the functional given by half the sum of the three positive roots above, more explicitly for $h = (h_1,h_2,h_3) \in H$, we have
$\tau(h) = h_1-h_3$.

Associated to this root space decomposition, we have the Weyl group $\mathcal{W} =S_3$, the group of symmetric permutations with 3 elements, and with the identification of $H$ as a subset of $\R^3$ as earlier, $\mathcal{W}$ acts on the Cartan subalgebra $H$ by permuting coordinates.

We also need to talk about the Cartan involution associated to $\gg$, start with the map 
$\theta: \gg \longrightarrow \gg$ given by $\theta(X) = -X^T$, then we have the associated Cartan decomposition given by $\gg = \k \oplus \p$, where $\k = \{Y \in \gg | \theta(Y) = Y \}$ and
$\p =  \{Y \in \gg | \theta(Y) = -Y \}$. Note that $H \subset \p$ and is a maximal subalgebra in $\p$.

We can now talk about the construction of standard parabolic subalgebras of $\gg$. First define the positive chamber associated to $\Sigma^+$, given by $H^+ = \{h \in H \hspace{0.2cm}|  \hspace{0.2cm} \beta(h) > 0 \hspace{0.2cm}  \forall \hspace{0.2cm} \beta \in \Sigma^+\}$. For
a subset $J \subset \Sigma^{++}$, define $H_J =  \cap_{\beta \in J} ker (\beta)$, recall that $H$ has a inner product induced from the Killing form $B$, so there in an orthogonal compliment $H^J$ associated to $H_J$, such that $H = H_J \oplus H^J$. Further let $\Sigma^J $ be the set of roots which are linear combinations of elements from $J$. Then the standard parabolic subalgebra $\p_j$ of $\gg$ is given by the following direct sum,

\begin{equation}
\p_J  = m_J\oplus H_J \oplus n_J
\end{equation}
where we denote $m = \k \cap Z(H)$, with $Z(H)$ being the centralizer of H in $\gg$ and
\begin{equation*}
n_J  =\bigoplus_{\beta \in \Sigma^+ \backslash \Sigma^J} g^{\beta} \hspace{0.2cm},\hspace{0.2cm} m_J  = m \oplus H^J \oplus \bigoplus_{\beta \in \Sigma^J} g^{\beta}
\end{equation*}

Thinking of the set $\Sigma_J^+  =  \Sigma^+ \backslash \Sigma^J$ as a set of positive roots for the adjoint action of H on $n_J$, one can define an analogous half the sum of positive roots
$\tau_J : H_J \longrightarrow \R$ by $\tau_J(h) = (1/2)\displaystyle \sum_{\beta \in \Sigma_J^+ } \beta(h)$.

An arbitrary subalgebra $\h$ of $\gg$ is called a parabolic subalgebra if it is conjugate to one of these standard parabolic subalgebras described above. Finally, a subgroup $P$ of $SL(3,\R)$ is called parabolic subgroup if it is the normalizer of a parabolic subalgebra $\h$ in $\gg$ under the adjoint action of $SL(3,\R)$ on its Lie algebra $\gg$. We will describe the standard parabolic subgroups of $SL(3,\R)$ in the next section along with their associated Langlands decomposition.

\pagebreak

We will end this section by explicitly describing the three standard parabolic subalgebra of $\gg$.

\begin{itemize}
\item[$\bullet$] start with $J = \emptyset$ the empty set, the associated minimal parabolic subalgebra is given by $\p_{\emptyset} = m_0 \oplus H \oplus n_0$., where $n_0 = \displaystyle \oplus_{\beta \in \Sigma^+} \gg_{\beta}$ and $m_0$ is the orthogonal complement(with respect to the Killing form) of $H$ in $Z(H)$ (the centralizer of $H$ in $\gg$). The associated standard minimal parabolic subgroup of $SL(3,\R)$ will be denoted by $P_0$ and will be described explicitly in the next section.

\item[$\bullet$] Let $J_1 = \{\alpha_{12}\}$, then associated maximal parabolic subalgebra of $\gg$ has a split component of real dimension one given by, $H_1 = \{(h,h,-2h) \in H\}$, with the associated map 
$\tau_1 : H_1 \longrightarrow \R$ , $\tau_1(h,h,-2h) = 3h$, the associated standard parabolic subgroup of $SL(3,\R)$ will be denoted by $P_1$ and will be described explicitly in the next section.

\item[$\bullet$] Let $J_2 = \{\alpha_{23}\}$, then associated maximal parabolic subalgebra of $\gg$ has a split component of real dimension one given by, $H_2 = \{(2h,-h,-h) \in H\}$, with the associated map 
$\tau_2 : H_2 \longrightarrow \R$ , $\tau_2(2h,-h,-h) = 3h$, the associated standard parabolic subgroup of $SL(3,\R)$ will be denoted by $P_2$ and will be described explicitly in the next section.

\end{itemize}

\subsection{Standard parabolic subgroup of G}
We will start with the Iwasawa decomposition of $G $ as, $G =KAN$, where $K$ as before denotes $SO(3)$, $A$ denote the set of 3 by 3 diagonal matrices with determinant one and positive diagonal entries and $N$ is the set of 3 by 3 upper triangular unipotent matrices.

Let $S$ denote space of 3 by 3 positive definite symmetric metrices with determinant one. Then G acts on $S$ in a natural fashion and one obtains the identification of $S$ with $G/K$, so most of the spectral analysis for the manifold $G/K$ is done by using this model space $S$, we will see some of this calculations soon.

Note that there are three $SL(3,\Z)$(from now on this subgroup will be denoted by $\Gamma$) conjugacy classes of parabolic subgroups of $SL(3,\R)$, from now on we will fix one standard representative in each class and they are described below.

\begin{itemize}
  \item[$\bullet$] There is the standard parabolic subgroup $P_0$ of rational rank two, which is just the subset of upper triangular matrices in $SL(3,\R)$. The associated Langlands decomposition is given by
  $P_0 = NAM_0$, where $M = \{\pm Id_{3 \times 3}\}$.

  \item[$\bullet$] There two more non-$\Gamma$-conjugate maximal parabolic subgroups $P_1,P_2$ both of rational rank one. Namely,
  
 \begin{equation}
P_1 = \Bigg\{  \begin{bmatrix}
a& b & c\\
d& e & f \\ 
0 & 0 & g \\
\end{bmatrix}\   \in SL(3,\R) \Bigg \}
 \end{equation}
 and
 \begin{equation}
P_2= \Bigg\{  \begin{bmatrix}
\alpha& \beta & \gamma\\
0& \delta & \omega \\ 
0 & \sigma & \tau \\
\end{bmatrix}\   \in SL(3,\R) \Bigg \}
 \end{equation}
 
Just like $P_0$, each $P_i$ has an associated Langlands decomposition $P_k = M_kA_kN_k$, $k = 1,2$.
\end{itemize}
where,
 \begin{equation*}
N_1 = \Bigg\{  \begin{bmatrix}
1& 0 & p\\
0& 1 & q \\ 
0 & 0 & 1 \\
\end{bmatrix}\   \in SL(3,\R) \Bigg \}
 \end{equation*}
  \begin{equation*}
M_1 = \Bigg\{  \begin{bmatrix}
a& b & 0\\
c& d & 0 \\ 
0 & 0 & \pm 1 \\
\end{bmatrix}\   \in SL(3,\R) \Bigg \}
 \end{equation*}
  \begin{equation*}
A_1 = \Bigg\{  \begin{bmatrix}
\alpha & 0 & 0\\
0& \alpha & 0 \\ 
0 & 0 & \alpha^{-2} \\
\end{bmatrix}\   \in SL(3,\R) , \alpha > 0 \Bigg \}
 \end{equation*}
and
\begin{equation*}
N_2 = \Bigg\{  \begin{bmatrix}
1& a & b\\
0& 1 & 0 \\ 
0 & 0 & 1 \\
\end{bmatrix}\   \in SL(3,\R) \Bigg \}
 \end{equation*}
  \begin{equation*}
M_2 = \Bigg\{  \begin{bmatrix}
\pm 1& 0 & 0\\
0& a & b \\ 
0 & c &  d \\
\end{bmatrix}\   \in SL(3,\R) \Bigg \}
 \end{equation*}
  \begin{equation*}
A_2 = \Bigg\{  \begin{bmatrix}
\alpha^2 & 0 & 0\\
0& \alpha^{-1} & 0 \\ 
0 & 0 & \alpha^{-1} \\
\end{bmatrix}\   \in SL(3,\R) , \alpha > 0 \Bigg \}
 \end{equation*}
 
 \pagebreak
 Associated to each of these three parabolic subgroups the symmetric space $G / K$ admits a horospherical decomposition, so that any $x \in G /  K$ can be represented as $x = (z,exp_{P_i}(h) ,w)  \in X^{P_i} \times A_i \times N_i$, where $exp_{P_i}: \H_{i} \longrightarrow A_i$ is the exponential map and $h \in \H_i$. Note that $h \in \H_i$ is uniquely determined by $x \in G /K$, let's denote this $h \in \H_i$ as $H_{P_i}(x)$
 
 Finally, note that associated to these three standard parabolic subgroups there are boundary symmetric spaces $X^{P_i} = M_i / (K \cap M_i)$ associated to $G/K$ and boundary locally symmetric spaces $X_{P_i} = (\Gamma \cap M_{P_i}) \backslash M_i / (M_i \cap K)$, associated to the locally symmetric space $X$, note that these spaces $X_{P_i}$ describe the geometry of the locally symmetric space $ X$ at infinity, (in particular they can be used to construct the Reductive Borel -Serre compactification of X). These 
 boundary components $X_{P_i}$ of course play an important role in the spectral theory of $X$, which we will discuss in the next section.

\subsection{Spectral decomposition of the Laplacian acting on \texorpdfstring{$\Gamma \backslash X$}{SPECDECOMP}} 

In this section we will study the decomposition of the space $L^2( \Gamma \backslash X)$  under the action of the Laplace operator $\Delta$, (there is a corresponding decomposition that can be shown  to be invariant under all $G$-invariant differential operators, but we will focus on the Laplacian $\Delta$).

Since, the space $X$ is noncompact with finite volume, it has both discrete as well as continuous spectrum, the later parametrized by the so-called \textbf{Eisenstein Series} which we will describe in this section.

In general, depending on the rational rank of a rational parabolic subgroup for $SL(3,\R)$, there are multiple such Eisenstein series, in the case of the space $X$, which is $\Gamma \backslash G /K$, there are only three $SL(3,\Z)$conjugacy classes of parabolic subgroups of $SL(3,\R)$, so there are essentially three different classes of Eisenstein series associated to the locally symmetric space $X$, which we will now describe. 

\subsection{Eisenstein Series associated to the parabolic subgroup \texorpdfstring{$P_0$}{SERIES}}

The general procedure for constructing Eisenstein is as follows, let $Q$ be a parabolic subgroup of 
$SL(3,\R)$, we will choose $Q$ to be one of the standard parabolic subgroups $P_i$ ($i =0,1,2$). Let $Q$ have the associated Langlands decomposition $Q = M_Q \times A_Q \times N_Q$, where $A_Q$ is the associated split component of $Q$, with Lie algebra $\a_Q$ and the associated exponential map $exp_Q : \a_Q \longrightarrow A_Q$. Further, let $X^Q$ be the associated boundary  symmetric space given by $M_Q / (K \cap M_Q)$  with $\pi_Q : X \longrightarrow X^Q$ be the projection map and $S_Q = (\Gamma \cap X^Q) \backslash X^Q$
.  Finally, let $\tau_Q$ be half the sum of positive roots corresponding to the adjoint action of $\a_Q$ on $n_Q = Lie(N_Q)$.  

Now choose an $L^2$-eigenfunction $\psi$ of the Laplacian on the associated boundary locally symmetric space $S_Q = (\Gamma \cap X^Q) \backslash X^Q$ and choose $\lambda \in \a_Q \otimes_{\R} \C$ such that $Re(\lambda) >> 0$. Then we define the Eisenstein Series, $E_Q(x,\lambda,\psi)$ ,a smooth $\Gamma$ invariant function on $G/K$, which will then descend to a smooth function on the space $X$.

\begin{equation}
E_Q(x,\lambda,\psi) = \sum_{\gamma \in (\Gamma \cap P) \backslash \Gamma} e^{(\tau_Q+\lambda)(H_Q(\gamma x))}\psi(\pi_Q(\gamma x))
\end{equation}

 If we have another parabolic subgroup $Q'$(not necessarily different from $Q$) with the associated Langlands decomposition 
 $Q' = M_{Q' } \times A_{Q'} \times N_{Q'}$, then the restriction of $E_Q(x,\lambda,\psi)$ along the parabolic subgroups $Q'$ is given by the integral.

 \begin{equation}
 E_{Q|Q'} (x,\lambda , \psi) = \int_{(\Gamma \cap N_{Q'}) \backslash N_{Q'}} E_{Q}(nx,\lambda,\psi) dn
 \end{equation}
 Where, $dn$ is a normalized Haar measure on $N_{Q'}$ so that the total volume of $N_{Q'}$ is equal to one.

 One of the main results of Langlands is the study of the restriction of Eisenstein Series along various parabolic subgroups of $G$. There are essentially two separate cases to be considered,

Coming back to the parabolic subgroup $P_0$, From its Langlands decomposition, it is clear that the boundary locally symmetric space associated to $P_0$ is just a one point space, $X_{P_0}  = \{*\}$, so that the choice for $\psi$ has to be the trivial identity function. Further, note that in this case, the Lie algebra of the associated split component $A_{P_0}$ can be identified in a natural way with the Cartan subalgebra $H$ of $\gg$, (this essentially follows from the construction of the standard parabolic subalgebra $\p_{\emptyset}$ in the previous section). Denote by $E_{P_0}(x,\lambda)$ the associated Eisenstein series. Also, the action of the Weyl group $S_3$ on $\a_{P_0}$ induces an action on $H^* \otimes_{\R} \C$.

Then the result of Langlands implies that 
$E_{P_0|P_i} (x,\lambda) = 0$ for $i=1,2$, whereas

\begin{equation} 
E_{P_0|P_0}(x,\lambda) = \sum_{w \in S_3} e^{(\sigma \lambda+\tau)(H(x)) } C(w,\lambda)
\end{equation}
where, $S_3$ is the permutation group on 3 letters and $C(w,\lambda)$ are the rank two scattering matrices associated to $SL(3,\Z) \backslash SL(3,\R) / SO(3)$. These scattering matrices can be calculated explicitly, for a proof of the following result refer to \cite{Miller2001}.
\begin{equation}
   C(w,\lambda) =  \displaystyle 
\prod_{1 \leq a < b \leq 3,w(a) > w(b)}   \frac{\Omega(\lambda_a -\lambda_b)}{\Omega(1+\lambda_a-\lambda_b)}
\end{equation}
Where $\Omega(s) = \pi^{-s/2}\Gamma(s/2)\zeta(s)$ with $s \in \C$.

\subsection{Reduction theory of \texorpdfstring{$\Gamma \backslash X$}{REDUC}}

We need to discuss the reduction theory of $\Gamma \backslash X$, which involves splitting this noncompact manifold into a compact core and certain ends that extend to infinity. More precise, the ends will be described by certain Siegel sets associated to the three standard parabolic subgroups of $SL(3,\R)$.

We need to note first that there is a well-defined map 
$I_Q : H \longrightarrow \a_Q$, where $Q$ is a parabolic subgroup
of $SL(3,\R)$ with Lie algebra associated to the split component given by $\a_Q$.

This map $I_Q$ is unique in the sense that, under the identification $H  = H_1 \oplus H_2$, the maps $I_{Q_i}$ are precisely the projection maps from $H$ to $H_i$(the Lie algebra of the split component associated to the standard parabolic subgroup $P_i(i=1,2)$ is just the projection map from $H$ to the subspace $H_i$ .

Further, observe that under the identification of $H$ with the subspace of $\R^3$ given by elements $ h   = (h_1,h_2,h_3)$ such that $h_1+h_2+h_3 = 0 $, the restriction of the Killing form to $H$ is just given by $B(h,k) = 6\langle h, k \rangle $, where
$\langle \bullet , \bullet \rangle $ is the standard Euclidean inner product from $\R^3$. As before, let $\tau_j$ denote the associated half the sum of positive roots associated to $P_i$. Then, under the Killing form there is a unique element $T_j \in H_j $ corresponding to $\tau_j$. Using the map $I$ defined as above, one can choose a unique element $T \in H$ such that $T$
satisfies the equation $I_{P_i}(T) = T_i$ for $ i =0,1,2$. In fact it can be easily calculated to be the vector $(1/4,0,-1/4)$.

Choose $r >> 0 $, and letting $T_{r} = rT \in H$, we first define a shifted Weyl chamber denoted by $H_Q^+(T_r)$ for a parabolic subgroup $Q$ of $SL(3,\R)$ given by,

\begin{equation}
H_Q^+(T_r)  = \{h \in H \hspace{0.1cm}| \beta(h) > \beta (I_Q(T_r)) \hspace{0.1cm} \forall \beta \in \Sigma^+ (Q,A_{Q}) \}
\end{equation}

we also define,

\begin{equation}
H_{Q,+}(T_r)  = \{h \in H \hspace{0.1cm}| B(I_Q(T_r)-h, v) \hspace{0.1cm} \forall \beta \in \Sigma^+ (Q,A_{Q}) , v\in \a_Q^+  \}
\end{equation}
Where $\a_Q^+$ is the positive Weyl chamber in $\a_Q$ associated to the set of positive roots $\Sigma^+(Q,A_Q)$. 

The first task is to describe these shifted chambers explicitly in the case of parabolic subgroups of $SL(3,\R)$. We start with the standard ones.

\begin{itemize}
    \item[$\bullet$] Let $Q  = P_0$, the standard minimal parabolic subgroup of $SL(3,\R)$. Then,
   \begin{equation*}
H_{P_0}^+(T_r)  = \{h =(h_1,h_2,h_3) \in \R^3 \hspace{0.1cm}|\sum_{i=1}^3 h_i = 0 , h_1 > h_2+r/4, h_2> h_3+r/4 \}
\end{equation*} 
and
\begin{equation*}
H_{P_0,+}(T_r)  = \{h =(h_1,h_2,h_3) \in \R^3 \hspace{0.1cm}|\sum_{i=1}^3h_i=0, h_1 \geq 0, h_3 \geq 0 , h_1 \leq h_3 +r/2 \}
\end{equation*} 

\item[$\bullet$] Let $Q  = P_1$, the  minimal parabolic subgroup of $SL(3,\R)$ defined earlier. Then,
   \begin{equation*}
H_{P_1}^+(T_r)  = \{(h,h,-2h) \in \R^3 \hspace{0.1cm}| 12h+r > 0 \}
\end{equation*} 
and
\begin{equation*}
H_{P_1,+}(T_r)  = \{(h,h,-2h) \in \R^3 \hspace{0.1cm}| 12h+r \leq 0 \}
\end{equation*} 

\item[$\bullet$] Let $Q  = P_2$, be the other minimal parabolic subgroup of $SL(3,\R)$ defined earlier. Then,
   \begin{equation*}
H_{P_2}^+(T_r)  = \{(2h,-h,-h) \in \R^3 \hspace{0.1cm}| 12h-r > 0 \}
\end{equation*} 
and
\begin{equation*}
H_{P_2,+}(T_r)  =  \{(2h,-h,-h) \in \R^3 \hspace{0.1cm}| 12h-r \leq 0 \}
\end{equation*}

\item[$\bullet$] Finally, let $Q$ be an arbitrary parabolic subgroup of $SL(3,\R)$, then there exists $k \in K$, such that 
$Q = kP_ik^{-1}$, where $P_i$ is one of the standard parabolic subgroups. The associated Langlands decomposition of $Q$ is given by $Q = (kM_{P_i}k^{-1}) kA_{P_i}k^{-1} (kN_{P_i}k^{-1})$, with the Lie algebra of the associated 
split component given by $\a_Q = Ad(k)\a_{P_i} $.

Then since, the Killing form on $\gg$ is invariant under all inner automorphism of the form $Ad(g) $, where $g \in SL(3,R)$, we have $H_{Q}^+(T_r)  =Ad(k) (H_{P_i}^+(T_r))$ and
$H_{Q,+}(T_r)  = Ad(k)(H_{P_i,+}(T_r))$.

We are now finally ready to state the decomposition of $\Gamma \backslash X$ coming from the reduction theory of proper parabolic subgroups of $SL(3,\R)$. Let

\begin{equation*}
    X_{T_r}  =  \bigcap_{Q} M_Q \times exp_Q(H_{Q,+}(T_r)) \times N_Q
\end{equation*}
where the intersection is over all proper parabolic subgroups
of $SL(3,\R)$.

 \begin{thm}\label{theorem:9_11}
 Let $T \in H$ be defined as before, and let $T_r = rT$, where $r >> 0$, then $X_{T_r}$ is a compact submanifold with corners of codimension zero and $\Gamma$-invariant, then 
 $X_r^T =  \Gamma \backslash X_{T_r}$ will serve as the compact core in $\Gamma \backslash X$.
 
 Furthermore, there exists compact submanifolds with corners, $M_i \subset \Gamma_{P_i} \backslash N_{P_i} \times X^{P_i}$ such that the subset $M_i \times  exp_Q(H_{Q}^+(T_r)) \subset \Gamma_{P_i} \backslash X$ is mapped injectively into
 $\Gamma \backslash X$ under the canonical projection map 
 $X \mapsto \Gamma \backslash X$, denoting these subspaces in $\Gamma \backslash X$ by $M^i$, we have the following disjoint decomposition of  $\Gamma \backslash X$

 \begin{equation}
     \Gamma \backslash X  = X_r^T \bigcup _{i = 0,1,2} M^i
 \end{equation}
 \end{thm}   
\end{itemize}

We will refer to $M^i$ as the Siegel end at Height $T_r$ associated to the parabolic subgroup $P_i$.

\pagebreak

\subsection{Scattering Flats in \texorpdfstring{$\Gamma \backslash X$}{FLATS}}

We produce scattering flats on the $\Q$-rank two locally symmetric space
$\Gamma \backslash X = SL(3,\Z) \backslash SL(3,\R) /SO(3)$. Let $P_0$ be the minimal parabolic subgroup of $SL(3,\R)$ consisting of upper triangular matrices. Denote by $A$ the subgroup of diagonal matrices in $SL(3,\R)$ with positive diagonal entries along with its Lie algebra $\a = \text{Lie}(A)$.

Since there is only one association and $\Gamma$ conjugacy class of rational parabolic subgroups of $G$, with our chosen representative $P_0$, the only scattering flats are scattering between the Siegel set $M^0$ and itself.

Recall the rational Langlands decomposition for $P_0$
given by $M_0 AN$, where $A$ and $N$ are as before and $M_0$ is a 2 point space consisting of $\pm$ 3 by 3 identity matrix. Then the associated horospherical 
decomposition of $X  = X_0 \times A \times N$, where
$X_0  = \{z_0\}$. Also, recall that the Lie algebra of A was earlier identified with the Cartan subalgebra H.

Then a scattering flat in $\Gamma \backslash X$ is of the form $S  = \pi(\tilde{S})$, where 
$S = (z_0, exp_{P_0}(t_1h_1+t_2h_2),n)$, where
$n \in N$, $h_i \in H^+$, and $\pi: X \longrightarrow \Gamma \backslash X$ is the canonical projection. Note that since $X_0$ is trivial according to the discussion we had in section 5.3 about parametrizing space for scattering flats, this scattering flat $S$ is the unique scattering flat with a uniquely associate sojourn vector keeping track of the subset of $S$ contained in $X^T_r$.

We can certainly produce different scattering flats which have different associated sojourn vector.

If we recall the definition(\autoref{def:6_3}) of a scattering flat coming from an abelian subalgebra $\a$ of $\gg$, it requires that a shifted Weyl chamber in $\a$ to be isometrically embedded into $\Gamma \backslash X$. The size of the shifting parameter is what essentially is the sojourn vector (modulo a correction term because the ends described in \autoref{prop:3_5} do in fact depend on parameter $r > 0$).

Choose $\gamma \in \Gamma , \gamma \notin P_0$, then the two parabolic subgroups $P_0$ and $P_1 = \gamma P_0 \gamma^{-1}$ are different, but $\Gamma$ conjugate, hence correspond to the same Siegel set $M^0$. However the shifting parameter $\beta(I_{P_1}(T_r))$ depends on the parameter $r$ as well as the parameter $\gamma$, which is clear from the calculation of the shifted Weyl chamber $H^+_{Q}(T,r)$ for a general parabolic subgroup of $SL(3,\R)$.

So, the sojourn vector associated to a scattering flat $S$ in $\Gamma \backslash X$ which is coming from a flat in $X$ corresponding to $P_0$ and $P_1$ (with respect to a common split compliment which depends on $\gamma$) will also depend on $\gamma$. This discussion points to the conclusion that there are countably many scattering flats in $\Gamma \backslash X$ parametrized by $N \times \{\gamma \in \Gamma | \gamma \notin P_0\}$.

We are now set to define the sojourn vector associated to a scattering flat, as before choosing $\gamma \in \Gamma, \gamma \notin P_0$, along with its Bruhat decomposition given by $\gamma = u_2z \gamma_a w u_1$ with $u_1,u_2 \in N$, $\gamma_a \in A$ and $w \in S_3$. 

\begin{defn}\label{def:9_13}
Choose $\gamma$ as above, then such a $\gamma$ gives rise to a family of scattering flats parametrized by $N$ and with a common associated sojourn vector given by $\text{log}(\gamma_a) \in H$.
\end{defn}

\subsection{Asymptotic expansion of \texorpdfstring{$C(w,\lambda)$}{SCATMATX}} 

In order to work out an asymptotic expansion for $C(\omega ,\lambda) $ we need to review the factorization technique for section 5.4, we proceed as follows.

Choose $\gamma \in \Gamma, \gamma \notin P_0$ and denote by $P_1 = \gamma P_0 \gamma ^{-1}$, then $P_0$ and $P_1$ are distinct associate minimal, $\Gamma$ conjugate parabolic subgroups corresponding the same Siegel set in the reduction theory in \autoref{prop:3_5}. Then $P_0$ and $P_1$ share a common rational split component $A$, such that the Lie algebra $\a = \text{Lie}(A)$ is the Cartan sub-algebra $H$ after a suitable conjugation. Let $P_i = MAN_i$ be their rational Langlands decomposition, with associated boundary symmetric spaces $X^i$ and the corresponding boundary locally symmetric spaces $S_i$. Then we know that there is a family of scattering flats in $\Gamma \backslash X$ parametrized by $N_0$ with a common sojourn vector $\text{log}(\gamma_a)$ and they are projection of flats in $X$ of the form
$S = (z,exp_{P_0}(t_1H_1+t_2H_2),n)$, with $z \in M$, $n \in N_0$.

Proceeding as in section 7, we can choose a maximal parabolic subgroup $Q$ containing both $P_0,P_1$ along with Langlands decomposition $Q = M_QA_Q N_Q$
and two $\Q$ rank one rational parabolic subgroups $Q_0,Q_1$ of $M_Q$, such that the boundary symmetric space $X^Q$ can be naturally identified with the upper half plane $\H$ and the associated boundary locally space $S_Q$ with the quotient $SL(2,\Z) \backslash \H$. Then the discussion at the end of section 5.5 clearly gives us the following result.

\begin{thm}\label{theorem:9_14}
The family of scattering flats described as above projects onto a family of scattering geodesics in
$SL(2,Z) \backslash \H$ running between ends corresponding to $Q_0$ and $Q_1$ with a common sojourn time given by $|\text{log}(\gamma_a)|$, where $|\bullet|$ is the norm on the Lie algebra $\a$ associated to the Killing form.
\end{thm}

In order to proceed further we need to talk about the scattering geodesics and associated scattering matrix for $SL(2,Z) \backslash \H$ and we do that now.

We need to talk about the scattering geodesics and scattering matrix on $X_1 = SL(2,\Z) \backslash \H $. Note that there is only one cusp in $X_1$ denoted by $k_{\infty}$, corresponding to the cusp at $\infty$. Then the procedure of chapter 3 gives all the scattering geodesics running between the end corresponding to $k_{\infty}$, in particular one obtains that such scattering geodesics are in one-one correspondence with the set $\Gamma_0 \backslash\{\gamma \in SL(2,\Z) , \gamma \notin U\}$, where, $\Gamma_{0}$ is the subgroup of $SL(2,\Z)$ generated by the map $z \mapsto z+1$ and $U$ is the set of upper triangular matrices in $SL(2,\R)$. Let $\T$ denote the set of associated normalized sojourn times(these don't depend on a cutoff parameter as opposed to the dependence we saw in section 2).

Note that associated to the cusp at $k_{\infty}$ in $X_1$, we have the associated Eisenstein series $E(z,s)$ given by,

\begin{equation}
    E(z,s) = \sum_{\sigma \in \Gamma_{0} \backslash SL(2,\Z)} (Im(\sigma z))^s
\end{equation}
where, $\Gamma_{0}$ is the subgroup of $SL(2,\Z)$ generated by the map $z \mapsto z+1$. Note that $E(z,s)$ is $SL(2,\Z)$ invariant by construction, therefore descends to well-defined function on $X_1$. Note that as we have already discussed in chapter 3, $E(z,s)$ converges uniformly and absolutely on compact subsets of the half plane $Re(s) > 1$ and defines a holomorphic function, and admits a meromorphic continuation to all $s \in \C$.

Further we have the following asymptotic expansion,

\begin{equation}
    E(z,s) = y^s+ C(s)y^{1-s}+O(e^{-c_1y}) \hspace{0.1cm}\text{as} \hspace{0.1cm} y \longrightarrow \infty
\end{equation}
Where $z = x+iy \in \H$, $C(s) = \displaystyle \frac{\Omega(2s-1)}{\Omega(2s)}$ is the uniquely associated scattering matrix to the cusp $k_{\infty}$ in $X_1$, with $\Omega(s)  = \pi^{-s/2} \Gamma(s/2)\zeta(s)$.

For a proof, refer to \cite[3.24]{Iwaniec2002}.

Note that, it follows then from the result of Guillemin in \autoref{theorem:2_4} that,

\begin{equation}
    C(s) = F(s)\sum_{T \in \T} e^{-T(1/2+is)}
\end{equation}

for $Im(s) \leq -3/2$, with $F(s)$ and the parameter \text{a} as defined in \autoref{theorem:2_4}.

Some comment need to be made as to the connection of the parameter $\text{a}$ to the parabolic subgroup $Q_0$ described earlier, note that one can show that it essentially happens that we have $a = I_{Q_0}(T_r)$, where $I_W(T_r)$ is the map introduced in section 6.4 for a general parabolic subgroup $W$ of $SL(3,\R)$.

We need a modified version of Guillemin's result which we now describe briefly. Let $r \in R$ and denote by $\Phi(r) = C(1/2+ir)$, then one can talk about the Fourier transform of $\Phi(r)$ in the sense of distributions. Then we have the following result ,

\begin{prop}\label{prop:6_1}
The singular support of the Fourier Transform of $\Phi$ is just the set of sojourn times $\T$.
\end{prop}
For a proof and further discussion, refer to example $\text(i2)$ in section 2 from \cite{Zelditch1992}.

Now we return to the rank two scattering matrix $C(w,\lambda)$ described in equation 6.5.1 with $w \in S_3$, $\lambda \in H^* \otimes _{\R} \C$ such that $Re(\lambda) > > 0$ and $a,b \in \{1,2,3\}$ chosen such that $a < b$ and $w(a) > w(b)$, then we have

\begin{equation}
    C(w,\lambda) = C(\tilde{s})
\end{equation}
with $\tilde{s} = (1/2)(\lambda_a-\lambda_b +1)$.

\pagebreak

We now choose $\lambda$ purely imaginary, in the sense that we choose $\lambda = (i\eta_1,i\eta_2,i\eta_3)$, with $\eta = (\eta_1,\eta_2,\eta_3) \in \R^3$, and define 
$S(w,\eta) = C(w,\lambda)$, then equation 6.5.6 translates to

\begin{equation}
    S(w,\eta) = \Phi(r)
\end{equation}
with $2r = \eta_a -\eta_b$.

We will now calculate the generalized Fourier transform of $S(w,\eta)$, for simplicity we choose $w = (12) \in S_3$, so $a = 1$ and $b =2$ works. The calculations for other choices of $w \neq Id$ can be worked out in a similar fashion as shown below.

\begin{align*}
    \hat{S}(w,\eta) & = \int_{-\infty}^{\infty}\int_{-\infty}^{\infty}\int_{-\infty}^{\infty}e^{-i\eta_1 \zeta_1}e^{-i\eta_2 \zeta_2}e^{-i\eta_3 \zeta_3}S(w,\eta) d\eta_1 d\eta_2 d\eta_3 \\
& =     \int_{-\infty}^{\infty}\int_{-\infty}^{\infty}\int_{-\infty}^{\infty}e^{-i\eta_1 \zeta_1}e^{-i\eta_2 \zeta_2}e^{-i\eta_3 \zeta_3}\Phi\Big (\frac{\eta_1 -\eta_2}{2} \Big) d\eta_1 d\eta_2 d\eta_3 \\
& = \int_{-\infty}^{\infty}\int_{-\infty}^{\infty}e^{-i\eta_1 \zeta_1}e^{-i\eta_2 \zeta_2} \Phi\Big (\frac{\eta_1 -\eta_2}{2} \Big)d\eta_1 d\eta_2\int_{-\infty}^{\infty}e^{-i\eta_3 \zeta_3} d\eta_3 \\
&   = (1/2)\delta_0(\zeta_3) \int_{-\infty}^{\infty}\int_{-\infty}^{\infty}e^{-iz_1 (\zeta_1/2+\zeta_2/2)}e^{-iz_2(\zeta_1/2 -\zeta_2/2)} \Phi(z_1)dz_1 dz_2 \\
& = (1/2)\delta_0(\zeta_3)\delta_0(\zeta_1/2 - \zeta_2/2)\hat{\Phi}(\zeta_1/2 +\zeta_2/ 2)
\end{align*}
Where $\delta_0(\zeta) $ is the Dirac Delta distribution on $\R$.

From this calculation along with the result of \autoref{prop:6_1} we get the following result.

\begin{thm}\label{theorem:9_21}
For $w = (12) \in S_3$, the Singular Support of the Fourier transform of $S(w,\eta)$ is precisely the set ${(T,T,0) \subset \R^3}$ where $T \in \T$. 
\end{thm}

Note that, unlike the rank one scattering matrices here the Singular Support of $S(w,\eta)$ only depends on the sojourn times of the projected geodesics and not on the full set of sojourn vectors in $\Gamma \backslash X$, which is precisely due to the factorization property.
\pagebreak

This equation along with Guillemin's Theorem and the discussion we had earlier can be summarized in the following result which we describe next.

\begin{thm}\label{theorem:9_22}
Let $\text{X} = SL(3,\Z) \backslash SL(3,\R) /SO(3)$ be the rational rank two locally symmetric space as before.  

A) For any $\gamma \in SL(3,\Z)$ such that $\gamma \notin P_0$, there is a continuous family of scattering flats in $X$ parametrized by the space of upper triangular unipotent matrices $N$ in $SL(3,\R)$ and all of these scattering flat have the same \textbf{sojourn vector} given by $\text{log}(\gamma_a) \in \a$ coming from the Bruhat decomposition of $\gamma$ as before.
We further have that any such given family of scattering flats projects onto a family of scattering geodesics into an associated Locally symmetric space which can be naturally identified with $SL(2,\Z) \backslash \H$ with a common normalized sojourn time given by $|\text{log}(\gamma_a)|$ where $|\bullet|$ is the norm on the Lie algebra $\a$ coming from the Killing form.

B) Denote by the set of such sojourn times by $\T $ as before. If $C(w,\lambda)$ denotes one of the rank two scattering matrices associated to $\Gamma \backslash X$ as before with $\lambda = (\lambda_1,\lambda_2,\lambda_3) \in H^* \otimes _{\R} \C$ with $Re(\lambda) >> 0$ and $a,b \in \{1,2,3\}$ chosen such that $a < b$ and $w(a) > w(b)$, we have for
$Im(\lambda_a -\lambda_b +1) \leq -3$ and
$\tilde{s} = (1/2)(\lambda_a -\lambda_b+1)$,
\end{thm}

\begin{equation}
    C(w,\lambda)  = C(\tilde{s})
\end{equation}

\begin{equation}
    C(w,\lambda)  = F(\tilde{s})\sum_{T \in \T} e^{-T(1/2+i\tilde{s})}
\end{equation}
where  
$F(s)$ is the function defined in \autoref{theorem:2_4}.

Note that equation 9.24 contains an asymptotic expansion for the scattering matrix which is very similar to what was worked out by Guillemin in \cite{Guillemin1976}, of course the one big difference is that the terms of this expansion only depend on the sojourn times of the projected scattering geodesics rather than the set of full sojourn vectors in $\Gamma \backslash X$.

\nocite{*}
\bibliographystyle{aomalpha}    
\bibliography{thesis-bib2}      

\end{document}